\documentclass{amsart}
\usepackage{amsmath}
\usepackage{amssymb}
\usepackage[all]{xy}
\usepackage{comment}
\usepackage{stackrel}

 \usepackage[notcite,notref]{showkeys}

\theoremstyle{plain}
\newtheorem{thm}{Theorem}[section]
\newtheorem{thm*}{Theorem}[section]
\newtheorem{cor}[thm]{Corollary}
\newtheorem{prop}[thm]{Proposition}
\newtheorem{lemma}[thm]{Lemma}

\newtheorem{lemma*}{Lemma}

\theoremstyle{definition}
\newtheorem{defn}[thm]{Definition}
\newtheorem{remark}[thm]{Remark}

\newtheorem*{remark*}{Remark}

\newtheorem{ex}[thm]{Example}
\newtheorem{notation}[thm]{Notation}

\newtheorem{question*}{Question}

\numberwithin{equation}{thm}

\newcommand{\bR}{\mathbb R}
\newcommand{\bS}{\mathbb S}

\newcommand{\bM}{\mathbb M}

\newcommand{\cN}{\mathcal N}

\def\Rad{\operatorname{Rad}\nolimits}
\def\Spec{\operatorname{Spec}\nolimits}
\def\Ker{\operatorname{Ker}\nolimits}

\def\Lie{\operatorname{Lie}\nolimits}
\def\Soc{\operatorname{Soc}\nolimits}

\def\Im{\operatorname{Im}\nolimits}
\def\Proj{\operatorname{Proj}\nolimits}
\def\dim{\operatorname{dim}\nolimits}

\def\Ind{\operatorname{Ind}\nolimits}

\def\Grass{\operatorname{Grass}\nolimits}
\def\LG{\operatorname{LG}\nolimits}

\def\Aut{\operatorname{Aut}\nolimits}
\def\Dim{\operatorname{Dim}\nolimits}
\def\Ad{\operatorname{Ad}\nolimits}
\def\Stab{\operatorname{Stab}\nolimits}

\newcommand{\bSoc}{\mathbb S\rm oc}
\newcommand{\bRad}{\mathbb R\rm ad}

\newcommand{\bG}{\mathbb G}

\newcommand{\bA}{\mathbb A}

\newcommand{\bP}{\mathbb P}
\newcommand{\bZ}{\mathbb Z}

\newcommand{\cC}{\mathcal C}

\newcommand{\bW}{\mathbb W}

\newcommand{\fb}{\mathfrak b}
\newcommand{\fc}{\mathfrak c}
\newcommand{\fp}{\mathfrak p}

\newcommand{\ff}{\mathfrak f}
\newcommand{\fg}{\mathfrak g}
\newcommand{\fh}{\mathfrak h}

\newcommand{\fz}{\mathfrak z}
\newcommand{\fu}{\mathfrak u}

\newcommand{\gl} {\mathfrak {gl}}

\newcommand{\fn}{\mathfrak n}
\newcommand{\fsl} {\mathfrak {sl}}

\newcommand{\fsp} {\mathfrak {sp}}

\newcommand{\fl}{\mathfrak l}

\newcommand{\ul}{\underline}

\def\pr{\operatorname{pr}\nolimits}

\def\id{\operatorname{id}\nolimits}
\def\Spec{\operatorname{Spec}\nolimits}
\def\sl2{\operatorname{SL_{2(2)}}\nolimits}
\def\Ga2{\operatorname{\mathbb G_{\rm a(2)}}\nolimits}
\def\SL{\operatorname{SL}\nolimits}
\def\GL{\operatorname{GL}\nolimits}

\def\Sp{\operatorname{Sp}\nolimits}

\def\HHH{\operatorname{H}\nolimits}
\def\Ext{\operatorname{Ext}\nolimits}

\def\Hom{\operatorname{Hom}\nolimits}

\def\red{\operatorname{red}\nolimits}

\newcommand{\wt}{\widetilde}

\newcommand{\Z}{\mathbb Z}

\newcommand{\bE}{\mathbb E}
\newcommand{\bV}{\mathbb V}

\newcommand{\bu}{\bullet}

\date\today

\begin{document}

 \title[Elementary subalgebras ]{Elementary subalgebras of Lie algebras}
 
 \author[Jon F. Carlson, Eric M. Friedlander, and Julia Pevtsova]
{Jon F. Carlson$^*$, Eric M. Friedlander$^{**}$, and 
Julia Pevtsova$^{***}$}

\address {Department of Mathematics, University of Georgia, 
Athens, GA}
\email{jfc@math.uga.edu}

\address {Department of Mathematics, University of Southern California,
Los Angeles, CA}
\email{ericmf@usc.edu, eric@math.northwestern.edu}

\address {Department of Mathematics, University of Washington, 
Seattle, WA}
\email{julia@math.washington.edu}


\thanks{$^*$ partially supported by the NSF grant DMS-1001102}
\thanks{$^{**}$ partially supported by the NSF grant DMS-0909314 and DMS-0966589}
\thanks{$^{***}$ partially supported by the NSF grant DMS-0800930 and DMS-0953011}



\begin{abstract} 
We initiate the investigation of  the projective 
varieties $\bE(r,\fg)$ of elementary subalgebras of 
dimension $r$ of a ($p$-restricted) Lie algebra $\fg$ 
for various $r \geq 1$.   These varieties $\bE(r,\fg)$ are 
the natural ambient varieties for generalized support
varieties for restricted representations of $\fg$.  
We identify these varieties in special cases, revealing
their interesting and varied geometric structures.   We 
also introduce invariants for a  finite dimensional $\fu(\fg)$-module
$M$, the local $(r,j)$-radical rank and local $(r,j)$-socle
rank, functions which are lower/upper semicontinuous
on $\bE(r,\fg)$.  Examples are given of $\fu(\fg)$-modules for
which some of these rank functions are constant.
\end{abstract}

\maketitle

\section{Introduction}

We say that a Lie subalgebra $\epsilon \subset \fg$  
of a $p$-restricted Lie algebra $\fg$ over a field $k$ 
of characteristic $p$ is {\it elementary} if it is abelian with trivial
$p$-restriction. Thus,  if $\epsilon$ has dimension 
$r$, then $\epsilon \simeq \fg_a^{\oplus r}$ 
where $\fg_a$ is the one dimensional Lie algebra of 
the additive group $\bG_a$.  This paper 
is dedicated to the study of  the projective variety 
$\bE(r, \fg)$ of elementary subalgebras 
of $\fg$ for some positive integer $r$ and its 
relationship to the representation theory of $\fg$. 

We have been led to the investigation of $\bE(r,\fg)$ 
through considerations of cohomology and modular 
representations of finite group schemes.   Recall that 
the structure of a
 restricted representation of $\fg$  on a $k$-vector space  
is equivalent to the structure of a module 
for the restricted enveloping algebra $\fu(\fg)$ of $\fg$ 
(a cocommutative Hopf algebra over $k$ of dimension 
$p^{\dim(\fg)}$).    A key precursor of this present work 
is the identification of the spectrum of the
cohomology algebra  $\HHH^*(\fu(\fg),k)$ with the 
$p$-nilpotent cone $\cN_p(\fg)$
achieved in \cite{FPar1}, \cite{Jens}, \cite{AJ84}, \cite{SFB2}.  
The projectivization of $\cN_p(\fg)$ equals $\bE(1,\fg)$.
More generally,
$\bE(r,\fg)$ is the orbit space under the evident $\GL_r$-action 
on the variety of $r$-tuples of commuting, 
linearly independent, p-nilpotent elements of $\fg$.    Our 
interest in  $\bE(r,\fg)$ and its close connections
with the representation theory of $\fg$ can be traced back 
through the 
work of many authors to the fundamental papers of Daniel Quillen  
who established the important 
geometric role that elementary abelian $p$-subgroups play in 
the cohomology  theory of  finite groups \cite{Q}.

It is interesting to observe that
the theory of cohomological support varieties for 
restricted $\fg$-representations 
(i.e., $\fu(\fg)$-modules)
as considered first in \cite{FPar2} has evolved into 
the more geometric study of $\pi$-points as 
introduced by the second and third authors in \cite{FP2}.  
This latter work closed a historical loop, 
relating cohomological considerations to earlier 
work on cyclic shifted subgroups as
investigated by Everett Dade \cite{Dade} and the first author \cite{C1}.

For $r > 1$ and $\fg$ the Lie algebra of a connected reductive group 
$G$, $\bE(r,\fg)$ is closely related to the spectrum of 
cohomology of the $r$-th Frobenius kernel $G_{(r)}$ of 
$G$ (see \cite{SFB1} for classical simple groups $G$;  
\cite{McN}, \cite{Sob} for more general types).   
Work of Alexander Premet 
concerning the variety of commuting, nilpotent pairs in $\fg$ \cite{Prem} 
gives  considerable information about $\bE(2,\fg)$.    
 Much less is known for larger $r$'s, although work 
in progress indicates the usefulness
 of considering the representation theory of $\fg$ 
when investigating the topology of $\bE(r,\fg)$. 
  
We consider numerous examples of restricted Lie 
algebras $\fg$ in Section \ref{Er},
and give some explicit computations of $\bE(r, \fg)$.  
Influenced by the role of maximal elementary abelian 
$p$-subgroups in the study of the cohomology
of finite groups, we are especially interested in  
examples of $\bE(r, \fg)$ considered in 
Section \ref{sec:max} for which $r$ is maximal 
among the dimensions of elementary subalgebras of $\fg$.  
For simple Lie algebras over a field of  characteristic 0, 
Anatoly Malcev determined this maximal dimension
\cite{Mal51} which is itself an interesting invariant 
of $\fg$.  Our computations verify that the Grassmann 
variety of $n$ planes in a $2n$-dimensional
$k$-vector space maps bijectively (via a finite, radicial
morphism) to $\bE({n^2}, \gl_{2n})$; similar results apply to
the computation of $\bE(n(n+1),\gl_{2n+1})$  and 
$\bE\left(\frac{(n+1)n}{2}, \fsp_{2n}\right)$.  As we point out 
in Section~\ref{sec:max}, these maps turn out to be 
isomorphisms of varieties.   The reader interested in the description of 
$\bE(r, \fg)$ for other types of simple Lie algebras $\fg$ can find them in a 
forthcoming paper \cite{PS}.  
We also provide some computations for restricted Lie algebras not arising from reductive groups.  

We offer several explicit motivations for considering 
$\bE(r,\fg)$ in addition to the fact that these projective varieties
are of intrinsic interest.  Some of these  motivations are pursued 
in Sections 3 and 4 where (restricted) representations
of $\fg$ come to the fore.  We point to the forthcoming paper 
\cite{CFP4}, which utilizes the discussion of this current work 
in an investigation of coherent sheaves
and algebraic vector bundles on $\bE(r,\fg)$.

\vspace{0.1in}
$\bu$ 
The varieties $\bE(r, \fg)$ are the natural ambient 
varieties in which to define generalized
support varieties for restricted representations 
of $\fg$ (as in \cite{FP3}).

\vspace{0.1in}
$\bu$ Coherent sheaves on $\bE(r,\fg)$ are naturally 
associated to arbitrary (restricted) representations of $\fg$
(see \cite{CFP4}).

\vspace{0.1in}
$\bu$  For certain representations of $\fg$ including those 
of constant Jordan type, the associated coherent sheaves are 
algebraic vector bundles on $\bE(r, \fg)$ (see \cite{CFP4}).

\vspace{0.1in}
$\bu$   Determination of the  (Zariski) topology of $\bE(r,\fg)$ 
is an interesting challenge which can be informed 
by the representation theory  of $\fg$.

\vspace{0.1in}
The isomorphism type of the restriction $\epsilon^*M$ of a 
$\fu(\fg)$-module $M$ to an elementary
subalgebra $\epsilon$ of dimension 1 is given by its Jordan 
type, which is a partition of the 
dimension of $M$.
On the other hand, the classification of indecomposable 
modules of  an elementary subalgebra of dimension 
$r > 1$ is a wild problem (except in the special case in 
which $r = 2 = p$), so that the isomorphism
types of $\epsilon^*M$ for $\epsilon \in \bE(r, \fg)$ do 
not form convenient invariants of a $\fu(\fg)$-module $M$.  
Following the approach undertaken in \cite{CFP2}, we consider  
the dimensions of the 
radicals and socles of such restrictions, 
$\dim \Rad^j(\epsilon^*M)$ and $\dim \Soc^j(\epsilon^* M)$,
for $\epsilon \in \bE(r, \fg)$ and any $j$ with 
$1 \leq j \leq (p-1)r$.  As we establish in
Section \ref{rad-soc}, these dimensions give 
upper/lower semi-continuous functions on
$\bE(r, \fg)$.  In particular, they lead to 
``generalized rank varieties" refining those introduced
in \cite{FP4}.  We achieve some computations of these 
generalized rank varieties $\bE(r, \fg)_M$ 
for two families of $\fu(\fg)$-modules $M$:  the $L_\zeta$ 
modules which play an important role in the theory of support varieties   
(see, for example, \cite[5.9]{Ben}) and induced modules.

One outgrowth of the authors' interpretation of cohomological support varieties
in terms of $\pi$-points (as in \cite{FP2}) is the 
identification of the interesting classes of 
modules of constant Jordan type and constant 
$j$-rank for $1 \leq j < p$ (see \cite{CFP1}).  
As already seen in \cite{CFP2}, this has a natural 
analogue in the context of elementary
subalgebras of dimension $r > 1$.  In Section 
\ref{sec:constant}, we give examples
of $\fu(\fg)$-modules of constant $(r,j)$-radical 
rank and of constant $(r,j)$-socle rank.
This represents a continuation of investigations initiated by the authors in 
\cite{CFP1}, \cite{FP4} (see also \cite{Ba11}, \cite{B10},  
\cite{Ben2}, \cite{BP12}, \cite{CF09}, \cite{CFS11}, \cite{F09}, and  others).  

Although we postpone consideration of Lie algebras over fields of
characteristic 0, we remark that much of the formalism
of Sections 1 and 3, and many of the
examples in Sections 2 are valid (and often easier)
in characteristic 0.   On the other hand,
some of our results and examples, particularly in
Section 4, require that $k$ have 
positive characteristic.

In a sequel to this work (see \cite{CFP4})  we show that $\fu(\fg)$-modules of constant 
$(r,j)$-radical rank and of constant $(r,j)$-socle rank determine vector 
bundles on $\bE(r,\fg)$.  Of particular interest are those $\fu(\fg)$-modules 
not equipped with large groups of symmetries.   We anticipate 
that the investigation of such modules may provide algebraic vector bundles with 
interesting properties.

Throughout, $k$ is an algebraically closed field of characteristic $p>0$.
All Lie algebras $\fg$ considered in this paper 
are assumed to be finite dimensional
over $k$ and $p$-restricted; a Lie subalgebra $\fh \subset \fg$ is
assumed to be closed under $p$-restriction.  
Without explicit mention to the contrary, all
$\fu(\fg)$-modules are finite dimensional.

We thank Steve Mitchell and Monty McGovern 
for useful discussions pertaining to 
the material in Section \ref{sec:max}.    
We also thank the referee for many useful comments and suggestions. 
 

\section{The subvariety $\bE(r,\fg)$ of $\Grass(r,\fg)$}
\label{Er}

We begin by formulating the definition of 
$\bE(r,\fg)$ of the variety of elementary subalgebras 
of $\fg$ and establishing the existence of 
a natural closed embedding of $\bE(r,\fg)$ into the 
projective variety $\Grass(r,\fg)$ of $r$-planes 
of the underlying vector space of $\fg$.  Once
these preliminaries are complete, we introduce 
various examples which  reappear 
frequently, here and in \cite{CFP4}.

Let $V$ be an $n$-dimensional vector space and $r < n$ 
a positive integer.  We consider the projective 
variety $\Grass(r, V)$ of $r$-planes of $V$.  
We choose a basis for $V, \ \{v_1,\ldots, v_n\}$;  
a change of basis has the effect of 
changing the Pl\"ucker embedding (\ref{pl}) by 
a linear automorphism of $\bP(\Lambda^r(V))$.
We represent a choice of basis $\{ u_1,\ldots,u_r\}$ 
for an $r$-plane $U \subset V$ by an 
$n\times r$-matrix $(a_{i,j})$, where 
$u_j = \sum_{i=1}^n a_{i,j}v_i$.  Let $\bM_{n,r}^\circ \subset
\bM_{n,r}$ denote the open subvariety of  the affine space
$\bM_{n,r} \simeq \bA^{nr}$ consisting of 
$n\times r$ matrices of (maximal) rank $r$ and set 
$\xymatrix@=18pt{p: \bM_{n,r}^\circ \ar[r]&  \Grass(r, V)}$ equal to 
the map sending a rank $r$ matrix $(a_{i,j})$ to
the $r$-plane spanned by 
$\{  \sum_{i=1}^n a_{i,1}v_i, \ldots,  \sum_{i=1}^n a_{i,r}v_i \}$.

We summarize a few useful, well known facts about 
$\Grass(r,V)$. Note that there is a natural (left) 
action of $\GL_r$ on $\bM_{n,r}$ via multiplication 
by the inverse on the right.

\begin{prop}
\label{note}
For any subset $\Sigma \subset \{ 1, \ldots, n \}$ of cardinality $r$, set $U_\Sigma \subset 
\Grass(r, V)$ to be the subset of those $r$-planes $U \subset V$ with a representing $n\times r$ 
matrix $A_U$ whose $r\times r$ minor indexed by $\Sigma$ (denoted by $\fp_\Sigma(A_U)$) 
is non-zero.    Then we have the following:
\begin{enumerate}
\item  
$\xymatrix@=12pt{p: \bM_{n,r}^\circ \ar[r]&  \Grass(r, V)}$ is a principal 
$\GL_r$-torsor, locally trivial in the Zariski topology;
\item
Sending an $r$-plane $U \in U_\Sigma$ to the unique $n\times r$-matrix $A^\Sigma_U$ 
whose $\Sigma$-submatrix (i.e., the $r\times r$-submatrix whose rows are those of $A^\Sigma_U$ 
indexed by elements of $\Sigma$) is the identity determines a section of $p$ over $U_\Sigma$:
\begin{equation}
\label{section}
\xymatrix@=12pt{s_\Sigma: U_\Sigma  \ar[r]& \bM_{n.r}^\circ;}
\end{equation}
\item
The Pl\"ucker embedding 
\begin{equation}
\label{pl}
\fp: \Grass(r, V) \hookrightarrow \bP(\Lambda^r(\bV)), \quad U \mapsto [\fp_\Sigma(A_U)]
\end{equation}
sending $U \in U_\Sigma$ to the ${n\choose r}$-tuple of $r\times r$-minors of $A_U^\Sigma$ is a 
closed immersion of algebraic varieties;
\item
$U_\Sigma \subset \Grass(r, V)$ is a Zariski open subset, the complement of the zero locus 
of $\fp_\Sigma$, and is isomorphic to $\bA^{r(n-r)}$.
\end{enumerate}
\end{prop}

Elementary subalgebras as defined below play the central role in what follows.   

\begin{defn}
An elementary subalgebra $\epsilon \subset \fg$ of dimension $r$ 
is a Lie subalgebra of dimension $r$ which is commutative and has $p$-restriction equal to 0.
We define
$$ \bE(r,\fg)  \ = \ \{ \epsilon \subset \fg: ~ \epsilon 
\text {~elementary subalgebra of dimension } r\}$$
\end{defn}

We denote by $\cN_p(\fg) \subset \fg$ the closed 
subvariety of $p$-nilpotent elements  of $\fg$ (that is, 
$\cN_p(\fg) = \{x \in \fg\, | \, x^{[p]} = 0\}$), by $\cC_r(\cN_p(\fg)) 
\subset (\cN_p(\fg))^{\times r}$ the variety of $r$-tuples
of $p$-nilpotent, pairwise commuting elements of $\fg$, 
and by $\cC_r(\cN_p(\fg))^\circ \subset \cC_r(\cN_p(\fg))$
the open subvariety of linearly independent $r$-tuples
of $p$-nilpotent, pairwise commuting elements of $\fg$.

\begin{notation} 
\label{not:action} For an algebraic group $G$ with Lie algebra $\fg = \Lie G$, we consistently use the adjoint action of 
$G$ on $\bE(r,\fg)$.  Explicitly, for an $r$-dimensional elementary subalgebra $\epsilon \subset \fg$, and for $g \in G$, we denote by $g \cdot \epsilon \in \bE(r, \fg)$ the $r$-dimensional elementary subalgebra defined as follows: 
\[
G \cdot \epsilon := \{ \Ad(g) x \, | \, x \in \epsilon\}
\]
Consequently, we use $G \cdot \epsilon$ to denote the orbit of $\epsilon \in \bE(r, \fg)$ under this action.  
\end{notation} 

\begin {prop}
\label{embed}
Let $\fg$ be a Lie algebra of dimension $n$, let $r$ be a positive integer, $1 \leq r < n$, 
and let   $\Grass(r, \fg)$ be the projective variety of $r$-planes of $\fg$, where we view $\fg$ as a vector space.  
There exists a natural cartesian square 
 \begin{equation}
 \label{sq}
\xymatrix{ 
\cC_r(\cN_p(\fg))^\circ \ar[d] \ar@{^(->}[r] & \bM_{n,r}^\circ \ar[d]^p\\
\bE(r,\fg)  \ar@{^(->}[r] & \Grass(r, \fg)
} \end{equation}
whose vertical maps are $\GL_r$-torsors locally trivial for the Zariski topology
and whose horizontal maps are closed immersions.  In particular, $\bE(r,\fg)$
has a natural structure of a projective algebraic variety, as a 
reduced closed subscheme of $Grass(r,\fg)$.

If $G$ is a linear algebraic group with $\fg = \Lie (G)$,
then $\bE(r,\fg) \ \hookrightarrow \  \Grass(r, \fg)$ is a $G$-stable 
embedding with respect to the adjoint action of $G$.
\end{prop}

\begin{proof}
The horizontal maps of (\ref{sq}) are the evident inclusions, the left vertical map is the
restriction of $p$.  Clearly, (\ref{sq}) is cartesian; in particular, 
$\cC_r(\cN_p(\fg))^\circ \subset \bM_{n,r}^\circ$ is stable under the action of $\GL_r$.   

To prove that $\bE(r,\fg) \ \subset \ \Grass(r,\fg)$ 
is closed, it suffices to verify for each $\Sigma$
that $(\bE(r,\fg) \cap U_\Sigma) \ \subset \ U_\Sigma$ 
is a closed embedding.   The restriction
of (\ref{sq}) above $U_\Sigma$ takes the form
 \begin{equation}
 \label{sqq}
\xymatrix{ 
\cC_r(\cN_p(\fg))^\circ \cap p^{-1}(U_\Sigma) \ar[d] \ar[r] & p^{-1}(U_\Sigma) \ar[d]^p 
\ar[r]^-{\sim} & U_\Sigma \times \GL_r \ar[d]^{pr}\\
\bE(r,\fg) \cap U_\Sigma \ar[r] & U_\Sigma \ar@{=}[r] & U_\Sigma
} 
\end{equation}
Consequently, to prove that  $\bE(r,\fg) \ \subset \ \Grass(r,\fg)$ is closed and that 
$\cC_r(\cN_p(\fg))^\circ \to \bE(r,\fg)$ is a $\GL_r$-torsor 
which is locally trivial for the Zariski topology
it suffices to prove that $\cC_r(\cN_p(\fg))^\circ \subset \bM_{n,r}^\circ$ is closed.

It is clear that $\cC_r(\cN_p(\fg)) \subset \bM_{n,r}$ is a closed subvariety since it is  
defined by the vanishing of the Lie bracket and the $p$-operator $(-)^{[p]}$ both of which 
can be expressed as polynomial equations on the matrix coefficients. Hence, 
$\cC_r(\cN_p(\fg))^\circ = \cC_r(\cN_p(\fg)) \cap \bM_{n,r}^\circ$  is closed in $\bM_{n,r}^\circ$.

If $\fg = \Lie(G)$, then the (diagonal) adjoint action 
of $G$ on $n\times r$-matrices $\fg^{\oplus r}$
sends a matrix whose columns
pair-wise commute and which satisfies the condition that $(-)^{[p]}$ vanishes on these columns
to another matrix satisfying the same conditions (since $\Ad: G \to \Aut(\fg)$ preserves both 
the Lie bracket and the $p^{th}$-power).  Thus, $\bE(r,\fg)$ is $G$-stable.
\end{proof}

\begin{remark}
\label{rem:cfp2} Let $V$ be a $k$-vector space of dimension $n$, and let $V^\# = \Hom_k(V,k)$ denote its linear dual.
Consider $\bV \equiv \Spec S^*(V^\#) \simeq \bG_a^{\times n}$, the vector group on the (based) 
vector space $V$. Then  $\Lie(\bV) \simeq \fg_a^{\oplus n}$ and we have an isomorphism of algebras
$$
\fu(\Lie \bV) \ \simeq \ \fu(\fg_a^{\oplus n}) \ \simeq \ k[t_1,\ldots,t_n]/(t_1^p,\ldots,t_n^p).
$$
Let $E = (\bZ/p)^{\times n}$ be an elementary abelian $p$-group of rank $n$ and choose an embedding 
of $V$ into the radical $\Rad(kE)$ of the group algebra of $E$ such that the composition with the 
projection to $\Rad(kE)/\Rad^2(kE)$ is an isomorphism.  This choice determines an isomorphism 
$$
\xymatrix@=18pt{\fu(\Lie(\bV))\ar[r]^-\sim & kE}.
$$
With this identification, the investigations of \cite{CFP2} 
are special cases of 
considerations of this paper.
\end{remark}

\begin{ex}
\label{ex:r=1}
For any (finite dimensional, $p$-restricted) Lie algebra, 
$$
\bE(1,\fg)\ \simeq \ \Proj k[\cN_p(\fg)]
$$
as shown in \cite{SFB2}, where $k[\cN_p(\fg)]$  is the (graded) coordinate
algebra of the $p$-null cone of $\fg$.
If $G$ is reductive with $\fg = \Lie(G)$ and if $p$ is good for $G$, then $\cN_p(\fg)$ is 
irreducible and equals the $G$-orbit $G \cdot \fu$ of the nilpotent radical 
of a specific parabolic subalgebra $\fp \subset \fg$ (see \cite[6.3.1]{NPV}).
\end{ex}

\begin{ex}
\label{ex:premet} 
 Let $G$ be a connected reductive algebraic group, let 
$\fg = \Lie G$, and assume that $p\geq h$, the Coxeter number of $G$, 
and that the derived subgroup of $G$ is simply connected. The 
assumption on $p$ implies that $\cN_p(\fg) = \cN(\fg)$, the null 
cone of $\fg$. Finally, we exclude the case when $G$ is of 
type $A_1$ to ensure that $\bE(2, \fg)$ is non-empty.

As shown by A. Premet in \cite{Prem},
$\cC_2(\cN(\fg))$ is equidimensional with irreducible components enumerated by the 
distinguished nilpotent orbits of $\fg$; in particular, 
$\cC_2(\cN(\gl_n))$ is irreducible.  This
easily implies that  $\bE(2,\fg)$ is an equidimensional 
variety, irreducible in the special case
$\fg = \gl_n$.   Since $\dim \bE(2,\fg) = \dim  \cC_2(\cN_p(\fg)) - \dim \GL_2$, 
$\dim \bE(2,\fg) = \dim [G,G] - 4$.   In particular, $\bE
(2,\gl_n)$  has dimension $n^2 - 5$ for $p\geq n$. 
\end{ex}

\begin{ex}
\label{ex:u3}
Assume that $p >2$. 
Let $ \fu_3 \subset \gl_3$ denote the Lie subalgebra of 
strictly upper triangular matrices and 
take $r=2$.  Then a 2-dimensional elementary Lie subalgebra 
$\epsilon \subset \fu_3$ is spanned by 
$E_{1,3}$ and another element $X \in \fu_3$ not a scalar 
multiple of $E_{1,3}$.  We can
further normalize the basis of $\epsilon$ by 
subtracting a multiple of $E_{1,3}$ from $X$,
so that $X = a_{1,2}E_{1,2} + a_{2,3}E_{2,3}$.    
Thus, 2-dimensional elementary Lie subalgebras
$\epsilon \subset \fu$ are parametrized by points 
$\langle a_{1,2},a_{2,3} \rangle 
\in \bP^1$, so that $\bE(2, \fu_3) \simeq \bP^1$.  

In this case, $\fu_3$ is the Lie algebra of the 
unipotent radical of the Borel subgroup 
$B_3 \subset \GL_3$ of upper triangular matrices.  
The adjoint action of $\GL_3$ on $\gl_3$ induces the action 
of $B_3$ on $\bE(2, \fu_3)$ since $B_3$ stabilizes $\fu_3$.
With respect to this action of $B_3$,
$\bE(2, \fu_3)$ is the union of an 
open dense orbit consisting of regular nilpotent
elements of the form $a_{1,2}E_{1,2} + a_{2,3}E_{2,3}$, 
with $a_{1,2} \not= 0 \not= a_{2,3}$;
 and two closed orbits.  The open orbit is isomorphic 
to the 1-dimensional torus $\bG_m \subset \bP^1$
 and the two closed orbits are single points $\{ 0 \}, \{ \infty \}$.
\end{ex}
We thank the referee for the following observation.
\begin{prop}
\label{prop:sub} Let $G$ be a reductive algebraic group, 
let $\fg = \Lie G$ be the Lie algebra of $G$. Let $r$ be the 
Lie rank of $\fg$.  
and assume $p\geq h$, where $h$ the Coxeter number of $G$. Let 
$\epsilon_{reg} \in \bE(r, \fg)$ be an elementary 
subalgebra containing a regular element of $\fg$. Then 
$ G \cdot \epsilon_{reg} \subset \bE(r, \fg)$ is an open orbit.
\end{prop}
\begin{proof}  Let $X$ be a  regular nilpotent element. 
Recall that the nilpotent part of 
centralizer of $X$ in $\fg$ is generated by $\langle X, X^2, 
\ldots, X^r\rangle$. Hence, there exists 
an elementary algebra $\epsilon_{reg}$ 
of dimension $r$ contaning $X$.  Let $Z$ be the complement of the regular 
nilpotent orbit in $\cN_{p}(\fg) = \cN(\fg)$ (that is, $Z$ is the closure 
of the subregular orbit). Observe that any $r$-tuple 
of nilpotent commuting matrices of $\fg$ containing a {\it regular } nilpotent element 
has to be conjugate to $(X, X^2, 
\ldots, X^r)$ under the action of $G \times \GL_r$. This implies that the diagram 
\eqref{sq} extends as follows:
 \begin{equation*}
 \label{sq2}
\xymatrix{ 
\cC_r(Z)^\circ \ar[d] \ar@{^(->}[r] &\cC_r(\cN_p(\fg))^\circ \ar[d] \ar@{^(->}[r] & \bM_{n,r}^\circ \ar[d]^p\\
\bE(r,\fg) - G \cdot \epsilon_{reg} \ar@{^(->}[r] & \bE(r,\fg)  \ar@{^(->}[r] & \Grass(r, \fg).
} \end{equation*}
Since $\cC_r(Z)^\circ$ is a closed $\GL_r$-stable subset of $\cC_r(\cN_p(\fg))^\circ$, we conclude that $G \cdot \epsilon_{reg}$ is open in $\bE(r, \fg)$. 
\end{proof} 
 
\begin{ex}
\label{ex:grnr}
We consider the algebraic group $G = \GL_n$ and some 
$r, \ 1 \leq r < n$.  Let $\fu_{r,n-r}
\subset \gl_n$ denote the Lie subalgebra of $n \times n$ 
matrices $(a_{i,j})$ with $a_{i,j}=0$ unless 
$1 \leq i \leq r, \ r+1 \leq j \leq n$.  Then 
$\fu_{r,n-r} \subset \gl_n$ is an elementary subalgebra 
of dimension $r(n-r)$.  The argument given in 
\cite[\S 5]{MP} applies in our situation to show 
that $\fu_{r, n-r}$ is a maximal elementary subalgebra 
(that is, not contained in any other elementary subalgebra).

Let $X \subset \bE({r(n-r)}, \gl_n)$ denote
the $\GL_n$-orbit of $\fu_{r,n-r}$ (as defined in Notation~\ref{not:action}).  
Let $P_r$ be the standard parabolic subgroup of $\GL_n$ 
defined by the equations 
$a_{i,j} = 0$ for $i>r, j\leq n-r$.  Since $P_r$ is the 
stabilizer of $\fu_{r,n-r}$ under the adjoint action of $\GL_n$,  
 $X = G\cdot \fu_{r,n-r} \simeq \GL_n/P_r \simeq \Grass(r,n)$.  
Since $X$ is projective, it is a closed $\GL_n$-stable 
subvariety of $ \bE({r(n-r)}, \gl_n)$.  
\end{ex}

We next give examples of $p$-restricted Lie algebras 
which are not the Lie algebras of algebraic groups.

\begin{ex}
\label{hoch}
Let $\phi: \gl_{2n}\to k$ be a semi-linear map (so that $\phi(av) = a^p\phi(v)$),
and consider the extension of $p$-restricted Lie algebras, 
split as an extension of Lie algebras  (see \cite[3.11]{FPar1}):
\begin{equation}
0 \to k \to {\wt \gl}_{2n} \to \gl_{2n} \to 0, \quad (b,x)^{[p]} = (\phi(x),x^{[p]}).
\end{equation}
Then $\bE({n^2+1}, {\wt \gl}_{2n})$ can be identified with the subvariety of $\Grass(n, 2n)$
consisting of those elementary subalgebras $\epsilon \subset \gl_{2n}$ of dimension $n^2$ 
such that the restriction of $\phi$ to $\epsilon$ is 0 (or, equivalently, such that $\epsilon$
is contained in the kernel of $\phi$).
\end{ex}

\begin{ex}
\label{ex:semi} (1). Consider the general linear group 
$\GL_n$ and let $V$ be the defining representation. Let $\bV$ 
be the vector group associated to $V$ as in Remark~\ref{rem:cfp2}. We set 
\begin{equation}
\label{eq:g1n}
\xymatrix@-0.8pc{G_{1,n}\ar@{=}[r]^-{\rm def} & \bV \rtimes \GL_n, & 
g_{1,n}\ar@{=}[r]^-{\rm def} & \Lie G_{1,n}}
\end{equation} 
Any subspace $\epsilon \subset V$ of dimension $r<n$ can be 
considered as an elementary subalgebra of $g_{1,n}$. 
Moreover,  the $G_{1,n}$-orbit  of $\epsilon \in \bE(r,\fg_{1,n})$ 
can be identified with $\Grass(r, V)$. 

(2). More generally, let $H$ be an algebraic group, 
$W$ be a rational representation of $H$, and 
$\bW$ be the vector group associated to $W$. Let 
$G \ \equiv \ \bW \rtimes H$, and let $\fh = \Lie H$. 
A subspace $\epsilon \subset W$ of dimension 
$r<\dim W$ can be viewed as an elementary 
subalgebra of $\fg$. Moreover,   
the $G$-orbit of  $ \epsilon \in \bE(r,\fg)$ can be identified 
with the $H$-orbit of $\epsilon$ in $\Grass(r,W)$.
\end{ex}

We conclude this section by giving a straightforward way to obtain additional computations 
from known computations of $\bE(r,\fg)$.  The proof is immediate.

\begin{prop}
\label{prop:prod}
Let $\fg_1, \fg_2, \ldots, \fg_s$ be finite dimensional 
$p$-restricted Lie algebras and let $\fg \ =
\fg_1 \oplus \cdots \oplus \fg_s$.  Then there is a 
natural morphism of projective varieties
\begin{equation}
\label{Eprod}
\xymatrix{\bE(r_1,\fg_1) \times \cdots \times \bE(r_s, \fg_s) \ar[r]& \bE(r,\fg), \quad r = \sum r_i,}
\end{equation}
sending $(\epsilon_1 \subset \fg_1,\ldots,\epsilon_s \subset \fg_s)$ to 
$\epsilon_1 \oplus \cdots \oplus \epsilon_s \subset \fg$.
Moreover, if $r_i$ is the maximum  of the dimensions of the elementary subalgebras of $\fg_i$
 for each $i, 1 \leq i \leq s$, then this morphism is bijective.
\end{prop}

\begin{cor}
\label{cor:prod-sl2}
In the special case of Proposition \ref{prop:prod} in which each $\fg_i \simeq \fsl_2$, 
$r_1= \cdots = r_s=1$, (\ref{Eprod}) specializes to
$$(\bP^1)^{\times r} \ \simeq \ \bE(r,\fsl_2^{\oplus r}).$$
\end{cor} 
\begin{proof} This follows from the fact that 
$\bE(1, \fsl_2) = \Proj k[\cN(\fsl_2)] \simeq \bP^1$ 
(see, for example, \cite{FP3}). 
\end{proof}


\section{Elementary subalgebras of maximal dimension}
\label{sec:max}

The study of maximal abelian subalgebras in complex 
semi-simple Lie algebras has a long history,
dating back at least to the work of Schur in the general 
linear case at the turn of last century \cite{Sch05}. 
The dimensions of maximal abelian subalgebras of a complex 
simple Lie algebra are 
known thanks to the classical work of Malcev \cite{Mal51}.  
Malcev's arguments apply to the positive characteristic case with little 
modification showing that the maximal dimensions he determined 
also give maximal dimensions of elementary subalgebras of simple 
Lie algebras of types A, B, C, D, E, F, G at least for $p$ good. In this paper, 
we reproduce this calculation for types A and C.
   
As pointed out to us by S. Mitchell,
our investigation of Lie algebras over
fields of positive characteristic is closely related to 
the study  Barry \cite{Bar79} who considered the
analogous problem of identifying maximal elementary 
abelian subgroups of Chevalley groups.
Subsequent work by Milgram and Priddy \cite{MP} in the case of 
the general linear groups 
guided some of our calculations.

The reader finds below consideration of  
$\bE(r,\fg)$ for several families of $p$-restricted 
Lie algebras $\fg$ and $r$ the maximal dimension of 
an elementary subalgebra of $\fg$. 
\begin{itemize}
\item
Heisenberg Lie algebras (Proposition~\ref{prop:extra-spec})
\item
The general linear Lie algebra  $\gl_n$ (Theorems~\ref{thm:sl2m} and \ref{thm:sl2m+1}).
\item
The symplectic Lie algebra $\fsp_{2n}$. (Theorem \ref{thm:sp2n}).
\item
The Lie algebra of a maximal parabolic of $\gl_n$ (Theorem \ref{prop:p1}).
\item
The Lie algebras of Example~\ref{ex:semi}(1) (Corollary \ref{cor:g1n}).
\end{itemize}

In what follows, we consider a connected reductive algebraic group $G$ over $k$.
We choose a Borel subgroup $B = U \cdot T \ \subset \ G$, thereby fixing a basis
of simple roots $\Delta \subset \Phi$ and the subset of positive roots $\Phi^+$.  
For a simple root $\alpha \in \Delta$, we denote 
by $P_\alpha$, $\fp_\alpha$, the corresponding standard maximal 
parabolic subgroup and its Lie algebra.   We write 
$$\fp_\alpha \ = \ \fh \oplus \sum_{\beta \in \Phi_I^- \cup \Phi^+} kx_\beta,$$ where 
$x_\beta$ is the root vector corresponding to the root $\beta$ and $\Phi_I$ is the root 
subsystem generated by the subset $\Delta \backslash \{ \alpha \}$.
We follow the convention in \cite[ch.6]{Bour} in the 
numbering of simple roots.  For $\fg = \Lie(G)$
we denote by $\fh \subset \fg$ the Cartan algebra given by $\fh = \Lie(T)$ and write
$\fg \ = \fn^- \oplus \fh \oplus \fn$, the standard triangular decomposition.

We begin by recalling the explicit nature of 
the {\it Heisenberg Lie algebras} which 
not only constitutes our first example but also 
reappear in the inductive analysis of other
examples. 

\begin{defn}  A ($p$-)restricted Lie algebra $\fg$ is a Heisenberg 
restricted Lie algebra if the center $\fz$  of $\fg$ 
is one dimensional, $\fg/\fz$ is an elementary Lie algebra
and if the $p$-power operation vanishes on $\fg$. 
\end{defn}

The requirement that the $p$-restriction map vanish on a
Heisenberg algebra means that only example in the case that 
$p =2$ is the trivial example: $\fg = \fz$.  More generally,  if 
$p= 2$ then any restricted Lie algebra with vanishing 
restriction map is an elementary algebra.  


Let $\fg$ be a Heisenberg restricted Lie algebra. Then $\fg$ admits a basis 
\begin{equation} 
\label{eq:basis}
\{x_1, \ldots x_{n-1}, y_1, \ldots y_{n-1}, y_n\}
\end{equation} 
such that $y_n$ generates the one dimensional center $\fz$ of $\fg$ and 
$$[x_i, x_j] = [y_i,y_j]=0, 
\quad [x_i, y_j] = \delta_{i,j}y_n \quad 1\leq i,j \leq n-1.$$
\noindent Let $W = \fg/\fz$, let 
$\phi: \fg \to W$ be the projection map, and let 
$\sigma: W \to \fg$ be a $k$-linear right splitting of $\phi$. 
For $x,y \in W$,   let $\langle x, y \rangle$ be the coefficient 
of $y_n$ in $[\sigma(x), \sigma(y)]  \in \fz = ky_n$.   
So defined, $\langle -, - \rangle$ gives  $W$ a symplectic 
vector space structure. 

We recall that a subspace $L$ of a symplectic vector space $W$  is said to be 
Lagrangian if $L$ is an isotropic subspace (i.e., 
if the pairing of any two elements of $L$ is 0) 
of maximal dimension. We denote by $\LG(n, W)$ 
the {\it Lagrangian Grassmannian}  of $W$, 
the homogeneous space parameterizing the Lagrangian 
subspaces of $W$. Note that, if $L$ is a Lagrangian subspace of 
$W = \fg/\fz$, for $\fg$ and $\fz$ as in the previous paragraph,
then the inverse image $\phi^{-1}(L) \subseteq \fg$ is an elementary
Lie algebra. 

\begin{prop}
\label{prop:extra-spec} 
Let $\fg$ be a Heisenberg restricted Lie algebra of 
dimension $2n-1$.
Equip $W =  \fg/\fz$ with the symplectic form as above.
\begin{enumerate} 
\item The maximal dimension of an elementary subalgebra of $\fg$ is $n$.
\item $\bE(n, \fg)\ \simeq \ \LG(n-1, W)$.
\end{enumerate}
\end{prop} 

\begin{proof} 
Let $\phi: \fg \to W = \fg/\fz$ be the projection map.  
Observe that if a subalgebra $\epsilon$ of $\fg$ is 
elementary then $\phi(\epsilon)$ is an isotropic linear subspace  of $W$.  
Since $\dim \phi(\epsilon) + \dim \phi(\epsilon)^\perp = \dim W$ (where $\phi(\epsilon)^\perp$ 
denotes the orthogonal complement with respect to the symplectic form) and 
$\phi(\epsilon) \subset \phi(\epsilon)^\perp$ since $\phi(\epsilon)$ is isotropic, 
we get that $\dim \phi(\epsilon) \leq (\dim W)/2 
= n-1$, and, consequently, $\dim \epsilon \leq n$.  
Moreover, the equality holds if and only if 
$\epsilon/\fz$ is a Lagrangian subspace of $W$. 
Hence, $\bE(n, \fg) \simeq \LG(n-1, W)$.
\end{proof}
 
\begin{ex}
\label{ex:extrasp} 
We give various Lie-theoretic contexts in which
the Heisenberg Lie algebras arise. In every case, assume that $p > 2$.

\begin{enumerate} \item 
Let $\fg = \fsl_{n+1}$, and 
let $\fp_J \subset \fg$ be the standard parabolic 
subalgebra defined by the subset 
$J = \{\alpha_2, \ldots, \alpha_{n-1}\}$ of simple roots, that is, 
$\fp_J = \fh \oplus \bigoplus\limits_{\alpha \in \Phi_J^- \cup \Phi^+} kx_\alpha$, where 
$\Phi_J$ is the root subsystem of $\Phi$ generated by the subset of simple roots $J$. 
Then the unipotent radical 
$\fu_J = \bigoplus\limits_{\alpha \in \Phi^+\backslash \Phi_J^+} kx_\alpha$ of $\fp_J$ is a Heisenberg restricted 
Lie algebra of dimension $2n-1$. 
In matrix terms, this is the subalgebra of 
strictly  upper triangular matrices with non-zero 
entries in the top row or the rightmost column.   
\item 
Let $\fg = \fsp_{2n}$. Let $\fp = \fp_{\alpha_1}$  be 
the maximal parabolic subalgebra corresponding 
to the simple root $\alpha_1$.   
Let 
$\gamma_n = 2\alpha_1 + \ldots + 2 \alpha_{n-1} + \alpha_n$ be the highest long root, and let 
further 
\begin{equation}
\label{eq:basis2}
 \beta_i = \alpha_1 + \alpha_2 + \ldots + \alpha_i, \quad \gamma_{n-i} = \gamma_n - \beta_i. 
 \end{equation} 
Then $\fu_{\alpha_1}$, the nilpotent radical of  
$\fp_{\alpha_1}$ is a Heisenberg Lie algebra, and the basis 
$\{x_{\beta_1}, \ldots, x_{\beta_{n-1}}, x_{\gamma_{n-1}}, \ldots, x_{\gamma_1}, x_{\gamma_n}\}$
satisfies the conditions required in (\ref{eq:basis}). 

\item  {\it Type $E_7$.}  Let $\fp = \fp_{\alpha_1}$. 
Then the  nilpotent radical of $\fp$ is a Heisenberg Lie algebra. .
\end{enumerate}
\end{ex} 

\begin{remark}
The referee has pointed out that all of the above 
examples fit into a general pattern. Let $G$ be a simple algebraic group.  
Suppose that $\alpha$ is the positive root of maximum height. 
If $\beta$ is any other positive root, then $(\beta, \alpha^{\vee})$ is one of 
~0, ~1 or ~2, and has value ~2 if and only if $\beta = \alpha$. 
Then the direct sum of the root subspaces of $\fg = \Lie(G)$ 
spanned by $x_\beta$ with $(\beta, \alpha^{\vee}) >0$ is a Heisenberg
restricted Lie algebra provided $p>2$.
\end{remark}

The following well known property of parabolic 
subgroups is used frequently.

\begin{lemma}
\label{lem:com} 
Let $G$ be a simple algebraic group  and  $P$  be a standard 
parabolic subgroup of $G$. Let $\fp = \Lie(P)$ 
and $\fu$ be the nilpotent radical of $\fp$. 
Unless $G$ is of type $A_1$ and $p=2$, we have $[\fu, \fp] \ = \fu$. 
\end{lemma}

\begin{proof}
Since $\fu$ is a  Lie ideal in 
$\fp$, we have $[\fu, \fp]\subset \fu$. For the opposite inclusion, 
it suffices to show that for any simple root $\alpha$ such that 
$x_\alpha \in \fu$, we have $x_\alpha \in [\fh, \fu]$. Except for the 
situation excluded in the statement of the lemma, we can always find a 
simple root $\beta$ such that the entry $\langle \alpha, \beta \rangle$ of 
the Cartan matrix of $\fg$ is non-zero. Hence, $[h_\beta, x_\alpha]$ is a non-zero
multiple of $x_\alpha$, and we conclude that $\fu \in [\fh, \fu]$. 
\end{proof} 

In the examples that follow, the closed subvariety $\bE(r,\fg) \  \subset \ \Grass(r,\fg)$
is a single orbit or a disjoint union of two orbits for $G$. Such an orbit $G \cdot \epsilon$ 
can be described set-theoretically via the orbit map $\pi:  G \to \bE(r,\fg)$, $ g \mapsto g \cdot \epsilon$.
In order to use this observation to identify $\bE(r,\fg)$ as a homogenous
space $G/\Stab_G(\epsilon)$ (or a disjoint union of two homogeneous spaces), 
we need to know that the orbit map is separable. The following remark addresses this issue.

\begin{remark}
\label{rem:sep} 
Let $G$ be an algebraic group and  $X$ be a $G$-variety, both defined over an algebraically closed field $k$. 
For $x \in X$, the orbit map $\pi_x:G\to G \cdot x \subset X$ determines a homeomorphism
$\overline \pi_x: G/G_x \to G\cdot x$ where $G_x$ is the (reduced) stabilizer of $x$.
This is an isomorphism of varieties if the map $\pi_x$ is separable (equivalently, if 
the tangent map $d\pi_x$ at the identity is surjective).  In \cite[3.7]{CFP4}  we 
show that when $p > 2h-2$ where $h$ is the Coxeter number of a semi-simple algebraic 
group $G$, the orbit map  $G \to G \cdot \epsilon \subset \Grass(r, \fg)$ under 
the adjoint action of $G$ on $\Grass(r, \fg)$ is separable. This implies that the homeomorphisms 
of \eqref{thm:sl2m}(3), \eqref{thm:sl2m+1}(3) and \eqref{cor:gln} are isomorphisms of varieties  
at least when $p>2n-2$; and that the homeomorphism of Theorem~\ref{thm:sp2n}  is an isomorphism 
at least for $p>4n-2$. 

We point out that in a forthcoming paper \cite{PS}, the authors show that the orbit map  
$G \to G \cdot \epsilon \subset \Grass(r, \fg)$ is always separable in types A, B, C, D removing 
the restriction on $p$. Hence, the maps in \eqref{thm:sl2m}(3), \eqref{thm:sl2m+1}(3) and 
\eqref{cor:gln} are, in fact, isomorphisms for any $p$.   
\end{remark}

We consider the special linear Lie algebra $\fsl_n=\Lie(\SL_n)$ in two parallel theorems, one for $n$ even and the other for $n$ odd.  
We denote by $\fu_n \ = \ \Lie(U)$ the nilpotent radical of the Borel subalgebra $\fb = \Lie(B)$. We also use the 
notation $P_{r,n-r}$, $\fp_{r,n-r}$, and $\fu_{r,n-r}$ for the  
maximal parabolic corresponding to the simple root $\alpha_r$, its Lie algebra, 
and its nilpotent radical.

The first parts of both Theorem \ref{thm:sl2m} and Theorem \ref{thm:sl2m+1}
 are well-known in the context of maximal elementary abelian 
subgroups in $\GL_n(\mathbb F_p)$ (see, for example, \cite{Goo} or \cite{MP}).  
We use the approach of \cite{MP} to compute conjugacy classes.  

\begin{thm}
\label{thm:sl2m}
  Assume $p>2$, and  $m \geq 1$.  
\begin{enumerate} 
\item The maximal dimension of an elementary abelian subalgebra  of $\fsl_{2m}$  is $m^2$. 
\item Any elementary abelian subalgebra  of dimension $m^2$ is conjugate to $\fu_{m,m}$, the nilpotent  
radical of the standard maximal parabolic $P_{m,m}$. 
\item There is a finite, radicial morphism  $\xymatrix@=12pt{\Grass(m,2m) \ar[r]& \bE({m^2}, \fsl_{2m})}$, inducing a 
homeomorphism on Zariski spaces; this morphism is an isomorphism if $p > 4m-2$.
\end{enumerate}
\end{thm} 

\begin{proof} 

We prove the following statement by induction: any elementary subalgebra of $\fsl_{2m}$ has dimension 
at most $m^2$ and any subalgebra of such dimension  inside the  nilpotent radical $\fu_{2m}$ (the subalgebra 
of strictly upper triangular $2m \times 2m$-matrices) must coincide 
with  $\fu_{m,m}$.  This implies claims (1) and (2) of the theorem. 

The statement is clear for $m=1$.  Assume it is proved for  $m-1$. Let $\epsilon$ be an elementary 
subalgebra of $\fsl_{2m}$. Since $\epsilon$ consists of nilpotent matrices, it can be conjugated into 
upper-triangular form by Engel's theorem. 
Let $J = \{\alpha_2, \ldots, \alpha_{2m-2}\}$ and let $\fu_J$ be the nilpotent radical of the standard 
parabolic $P_J$ determined by $J$.  Since $[\fu_{2m}, \fu_J] \subset \fu_J$, 
this is a  Lie ideal in $\fu_{2m}$. 

We consider extension 
\[\xymatrix{0 \ar[r] & \fu_J \ar[r] & \fu_{2m} \ar[r] & \fu_{2m}/\fu_J \simeq \fu_{2m-2} \ar[r] & 0.}\] 
Pictorially, the Lie algebras can be represented as follows, 
where $\fu_J$ is in the positions marked by $*$ in the first array and 
$\fu_{2m-2}$ is isomorphic to the Lie algebra with the positions marked by $*$ in the second.  
\[
\fu_J: \quad 
\begin{array}{ccccc|c} 
0&*&*& \dots  &*&* \\ 
\hline 
0&0&0& \cdots &0&*\\
0&0&0&  \cdots&0&*\\
.&.&.& \ldots &&*\\
&&&&&* \\
0&0&0& \cdots &0&*\\
0&0&0& \cdots &0&*\\
0&0&0& \cdots &0&0  
\end{array};
\qquad 
\fu_{2m-2}: \quad 
\begin{array}{cccccc|c}
0&0&0&0& \dots  &0&0 \\
\hline
0&0&*&*& \cdots &*&0\\
0&0&0&*&  \cdots&*&0\\
.&.&.&.& \ldots &*&0\\
&&&&&&0 \\
0&0&0&0& \cdots &*&0\\
0&0&0&0& \cdots &0&0\\
0&0&0&0& \cdots &0&0
\end{array}
\]
By induction, the dimension of the projection of $\epsilon$ onto $\fu_{2m-2}$ is at most $(m-1)^2$, and this 
dimension is attained if and only if the image of $\epsilon$ under the projection is the subalgebra of 
$\fu_{2m-2}$ of all  block matrices  of the form 
$\begin{pmatrix} 
0&{\bf A}\\
0&0 
\end{pmatrix}$, where $\bf A$ is an $(m-1)\times(m-1)$ matrix.  Since $\fu_J$ is 
a Heisenberg Lie algebra of dimension $4m-3$ (see Example~\ref{ex:extrasp}(1)), 
Proposition~\ref{prop:extra-spec} implies that the maximal elementary 
subalgebra of $\fu_J$ has dimension $2m-1$.  Hence, $\dim \epsilon \leq (m-1)^2 + 2m-1 = m^2$.   

Now let's assume that $\epsilon$ has the maximal dimension $m^2$ and is upper-triangular. Our goal is to show that $\epsilon = \fu_{m,m}$. 
The argument in the previous paragraph implies that every element in $\epsilon \subset \fsl_{2m}$ has the form 
\begin{equation} 
\label{eq:genform}
\begin{pmatrix} 
0&{\bf v_2}&{\bf v_1}& *\\
0&0&{\bf A}&{\bf w_1}\\
0&0&0 & {\bf w_2}\\
0&0&0 &0
\end{pmatrix}  
\end{equation}
for some ${\bf v_i}, ({\bf w_i})^T \in k^{m-1}$. 

Let $\begin{pmatrix} 
0&{\bf v^\prime_2}&{\bf v^\prime_1}& *\\
0&0&0&{\bf w^\prime_1}\\
0&0&0 & {\bf w^\prime_2}\\
0&0&0 &0
\end{pmatrix}  
$  be an element  in $\epsilon \cap \fu_J$. Taking a bracket of this element with a general element in 
$\epsilon$ of the form as in \eqref{eq:genform}, we get  
\[\begin{pmatrix} 
0&0& {\bf v^\prime_2}{\bf A}& *\\
0&0&0&{\bf A}{\bf w^\prime_2}\\
0&0&0 &0\\
0&0&0 &0
\end{pmatrix}.  
\]
The assumption that $\epsilon$ has maximal dimension $m^2$ implies that for any $m-1 \times m-1$ matrix $A$, 
there is an element in $\epsilon$ of the form \eqref{eq:genform}. Since $\epsilon$ is abelian, we conclude 
that ${\bf v^\prime_2}{\bf A} =0$, ${\bf A}{\bf w^\prime_2}=0$ 
for any ${\bf A} \in M_{m-1}$. 
Hence,  ${\bf v^\prime_2}=0$, ${\bf w^\prime_2}=0$ which implies that $\epsilon \cap \fu_J \subset \fu_{m,m}$.   
Moreover, for the dimension to be maximal, we need $\dim \epsilon \cap \fu_J = 2m-1$. Hence, for 
any ${\bf v_1}, ({\bf w_1})^T \in k^{m-1}$, the matrix 
\[\begin{pmatrix} 
0& 0 &{\bf v_1}& 0\\
0&0& 0 &{\bf w_1}\\
0&0&0 & 0 \\
0&0&0 &0
\end{pmatrix}  
\] is in $\epsilon$. 

It remains to show that for an arbitrary element of $\epsilon$,  necessarily of the form \eqref{eq:genform}, 
we must have  $ {\bf v_2} =0, {\bf w_2} =0$.  We prove this by contradiction.  Suppose $\begin{pmatrix} 
0&{\bf v_2}&{\bf v_1}& *\\
0&0&{\bf A}&{\bf w_1}\\
0&0&0 & {\bf w_2}\\
0&0&0 &0
\end{pmatrix}  
 \in \epsilon$ with ${\bf v_2} \not = 0$. Subtracting a multiple of $\begin{pmatrix} 
0& 0 & 0 & 1\\
0&0& 0 &0\\
0&0&0 & 0 \\
0&0&0 &0
\end{pmatrix}  
$,
 which is necessarily in $\epsilon$, 
 we get that $M = \begin{pmatrix} 
0&{\bf v_2}&{\bf v_1}& 0\\
0&0&{\bf A}&{\bf w_1}\\
0&0&0 & {\bf w_2}\\
0&0&0 &0
\end{pmatrix}$ belongs to  $\epsilon$.  Since ${\bf v_2} \not = 0$, we can find a vector 
$({\bf w_1})^T \in k^{m-1}$ such that ${\bf v_2} \cdot ({\bf w_1})^T \not = 0$.  As observed above, we have $M^\prime = \begin{pmatrix} 
0& 0 & 0 & 0\\
0&0& 0 &({\bf w_1})^T\\
0&0&0 & 0 \\
0&0&0 &0
\end{pmatrix}  
$  in $\epsilon$. Therefore, $[M, M^\prime]$ has a non-trivial entry ${\bf v_2} \cdot ({\bf w_1})^T$ in 
the $(1, 2m)$ spot which contradicts commutativity of $\epsilon$.    Hence, ${\bf v_2}=0$. 
Similarly, ${\bf w_2}=0$.  This finishes the proof of the claim.

To show (3), let $\wt P$ denote the stabilizer of $\fu_{m,m}$  under the adjoint action 
of $\SL_{2m}$, so that $\SL_{2m}/\wt P \simeq \SL_{2m}\cdot \fu_{m,m}$.   By (2) and the fact 
that $P_{m,m}$ normalizes its unipotent radical $U_{m,m}$, and, hence, stabilizes $\fu_{m,m}$, the
orbit map $\SL_{2m} \to  \SL_{2m}\cdot \fu_{m,m} =  \bE(m^2,\fsl_{2m})$ factors as
$\SL_{2m} \to  \SL_{2m}/P_{m,m} \to   \SL_{2m}/\wt P$.
Since $P_{m,m}$ is maximal among (reduced) algebraic subgroups of $\SL_{2m}$, we conclude that
$\wt P_{\red} = P_{m,m}$.  Consequently, we conclude that
$$\Grass(m,2m) = \SL_{2m}/P_{m,m} \  \to \SL_{2m}/\wt P = \bE(m^2,\fsl_{2m})$$
is a torsor for the infinitesimal group scheme $\wt P/P_{m,m}$ and thus is finite and  radicial.  

The second  assertion of (3) (that the map $\Grass(m,2m)  \to  \bE(m^2,\fsl_{2m})$ is an isomorphism for $p > 4m-2$), 
is verified in \cite[3.7]{CFP4} as explained in Remark \ref{rem:sep}. 
\end{proof}

\begin{thm}
\label{thm:sl2m+1}
Assume $m>1$, $p>2$. 
\begin{enumerate} 
\item The maximal dimension of an elementary abelian subalgebra  of $\fsl_{2m+1}$  is $m(m+1)$. 
\item There are two distinct conjugacy classes of such elementary subalgebras, represented 
by $\fu_{m, m+1}$ and $\fu_{m+1,m}$.   
\item There is a finite radicial morphism 
\[\xymatrix@=18pt{\Grass(m,2m+1) \sqcup \Grass(m,2m+1) \ar[r]& \bE(m(m+1),\fsl_{2m+1})}\]
inducing a homeomorphism on Zariski spaces; this morphism is an isomorphism for $p > 4m$.
\end{enumerate} 
\end{thm}

\begin{proof} Let $\fu_3$ be the Heisenberg Lie algebra of strictly upper-triangular 
$3\times 3$ matrices. By Proposition~\ref{prop:extra-spec}, $\bE(2, \fu_3) \simeq 
\LG(1, 2) \simeq \bP^1$, and the maximal dimension is $2$.   In the following complete 
list of maximal elementary subalgebras of $\fu_3$, we separate the algebras $\fu_{1,2}$ 
and $\fu_{2,1}$ for easy referencing later in the proof. 
\begin{itemize}
\item[$\bu$]
$\fu_{1,2}=\left\{ \begin{pmatrix} 0&a&b \\0&0&0\\0&0&0  
\end{pmatrix} \, | \, a,b \in k \right\}$, 
\item[$\bu$] $\fu_{2,1}=\left\{ \begin{pmatrix} 0&0&b \\0&0&a\\0&0&0  
\end{pmatrix} \, | \, a,b \in k \right\}$,  
\item[$\bu$] a one-parameter family  
$ \left\{ \begin{pmatrix} 0&a&b \\0&0&xa\\0&0&0  
\end{pmatrix} \, | \, a,b \in k \right\}$ for a fixed  $x \in k^*$.
\end{itemize}
We prove the following statements by induction: For any $m>1$, 
an elementary subalgebra of $\fsl_{2m+1}$ 
has dimension at most $m(m+1)$. 
Any subalgebra of such dimension  inside $\fu_{2m+1}$ must 
coincide either with  $\fu_{m,m+1}$ or $\fu_{m+1,m}$. 
This implies (1) and (2).

Base case: $m=2$.   Any elementary subalgebra can be 
conjugated  to the upper-triangular form. So it suffices to 
prove the statement for an  elementary subalgebra 
$\epsilon$ of $\fu_5$, the Lie algebra of strictly 
upper triangular $5\times 5$ matrices. 
As in the proof of Theorem~\ref{thm:sl2m}, we 
consider  a short exact sequence of Lie algebras
\[\xymatrix{ 0 \ar[r]&\fu_J \ar[r] & \fu_5 \ar[r]^{\pr} & \fu_3 \ar[r]& 0} \]
where $J = \{\alpha_2, \alpha_3\}$ (and, hence, $\fu_J 
\subset \fu_5$ is the subalgebra of upper 
triangular matrices with zeros 
everywhere except for  the top row and  the rightmost column).   
Since $\dim (\pr(\epsilon)) \leq 2$ by the remark above, 
and  $\dim (\epsilon \cap \fu_J) \leq 4$ by Proposition~\ref{prop:extra-spec}(1), 
we get that $\dim \epsilon \leq 6$. For the equality 
to be attained, we need $\pr(\epsilon)$ to be one of the 
two dimensional elementary  subalgebras  listed above. 
If $\pr(\epsilon) = \fu_{2,1}$ then arguing exactly as in 
the proof for the even dimensional case, we 
conclude that $\epsilon = \fu_{3,2}  \subset \fu_5$. 
Similarly, if $\pr(\epsilon) = \fu_{1,2}$, 
then $\epsilon = \fu_{2,3}$.  We now assume that 
\[
\pr(\epsilon) =   \{ \begin{pmatrix} 0&a&b \\0&0&xa\\0&0&0  
\end{pmatrix} \, | \, a,b \in k \}.
\]
Let $A^\prime = \begin{pmatrix} 0&a_{12}&a_{13}&*&* \\0&0&0&0&*\\0&0&0&0&a_{35}\\0&0&0&0&a_{45}\\0&0&0&0&0  
\end{pmatrix} \in \epsilon \cap \fu_J$, and  let  
$A = \begin{pmatrix} 0&*&*&*&* \\0&0&a&b&*\\0&0&0&xa&*\\0&0&0&0&*\\0&0&0&0&0  
\end{pmatrix} \in \epsilon$.  
Then 
\[ [A^\prime, A] = 
\begin{pmatrix} 0&0&aa_{12}&xa a_{13}+ba_{12}&* \\0&0&0&0&-aa_{35}-ba_{45}\\0&0&0&0&-xaa_{45}\\0&0&0&0&0\\0&0&0&0&0  
\end{pmatrix} 
\]
Since $\epsilon$ is abelian, and since the values  
of $a,b$ run through all elements of $k$, 
we conclude that  $a_{12}=a_{13}=a_{35}=a_{45}=0$. 
Therefore, $\dim \epsilon \cap \fu_J \leq 3$ 
and $\dim \epsilon \leq 5$.  Hence, the maximum is 
not attained in this case.  This finishes the proof in the base case $m=2$. 

We omit the induction step since it is very similar 
to the even dimensional case proved in Theorem~\ref{thm:sl2m}.  

To prove (2), we 
observe that $\fu_{m,m+1}$ and $\fu_{m+1,m}$ are not 
conjugate under the adjoint action of $\SL_{2m+1}$ since  
their nullspaces in the standard representation of $\fsl_{2m+1}$ have different dimensions. 

Finally, statement (3) follows from (1) and (2) as 
in the end of the proof of Theorem \ref{thm:sl2m}.
\end{proof} 

We make the immediate observation that the results of Theorems~\ref{thm:sl2m} and \ref{thm:sl2m+1} 
apply equally well to $\gl_n$. 
\begin{cor} 
\label{cor:gln} Assume $p>2$. \begin{enumerate} 
\item The maximal dimension of an elementary abelian subalgebra of $\gl_n$ is $\lfloor \frac{n^2}{4} \rfloor$.
  \item For any $m \geq 1$, there is a finite radicial morphism   
	\[
\xymatrix@=15pt{
\Grass(m, 2m) \ar[r]&\bE(m^2,\gl_{2m})
}
\]
  inducing a homeomorphism on Zariski spaces; this 
	morphism is an isomorphism for $p > 4m-2$.
\item For any $m \geq 2$,  there is a finite radicial morphism
\[
\xymatrix@=15pt{\Grass(m,2m+1) \sqcup \Grass(m,2m+1) \ar[r]& \bE(m(m+1), \gl_{2m+1})}
\]
inducing a homeomorphism on Zariski spaces; this morphism is an isomorphism for $p > 4m$.
\end{enumerate}
\end{cor}

\begin{remark}
In the case $n=3$, excluded above, the variety $\bE(2, \gl_3)$ is irreducible (see Example~\ref{ex:gl3}). 
\end{remark}

To make analogous calculations in the symplectic case, we need the 
following technical observation.

\begin{lemma}
\label{lem:sympl} Let $\epsilon$ be an elementary subalgebra of the symplectic Lie algebra $ \fsp_{2m}$. 
There exists an element $g \in \Sp_{2m}$ such that 
$g \epsilon g^{-1}$ belongs to the nilpotent radical of the standard Borel subalgebra of $\fsp_{2m}$.
\end{lemma}

\begin{proof}
Let $V$ be a $2m$-dimensional symplectic space with a basis 
$\{x_1, \ldots, x_m, y_m, \ldots y_1\}$ 
such that the symplectic form with respect to this basis has the standard matrix 
$S = \begin{pmatrix} 0 & I \\ -I & 0 \end{pmatrix}$.
A complete isotropic flag is a nested sequence of subspaces 
of the form: 
\[
0 \subset V_1\subset V_2\subset \ldots \subset V_m = V_m^\perp 
\subset V_{m-1}^\perp \subset \ldots 
\subset V_1^\perp \subset V
\]
such that $\dim V_i = i$. The condition that $V_i \subseteq V_i^\perp$
implies that each $V_i$ is isotropic. The standard Borel subalgebra $\fb$ of $\fsp_{2m}$
(such as in \cite[12.5]{EW}) 
is characterized as the stabilizer of the standard complete isotropic flag in $V$,
meaning the flag with $V_i$ spanned by $\{x_1, \ldots, x_i\}$ (so that
$V_i^\perp$ is spanned by $\{x_1, \ldots, x_n,y_n, \ldots, y_{n-i-1}\}$). 
Thus, each $V_i$, as given, has the 
property that $\fb V_i \subseteq V_i$. 
Any two complete isotropic flags are conjugate by an element of $\Sp_{2n}$.
Therefore if we show that the subalgebra $\epsilon$ stabilizes a complete isotropic
flag, then some conjugate of $\epsilon$ is contained in a standard Borel 
subalgebra of $\fsp_{2m}$, as asserted.

Constructing a complete isotropic flag that is invariant under $\epsilon$
is a straightforward inductive exercise. We begin with $i = 0$. Assume for 
some $i$ an isotropic $\epsilon$-invariant subspace 
$V_i \subseteq V_i^\perp$ has been constructed.
Choose $V_{i+1}$ to be any subspace
such that $V_i \subset V_{i+1}$ and $V_{i+1}/V_i$ is an $\epsilon$-invariant subspace 
of dimension one in $V_i^\perp/V_i$. Since $\epsilon$ is an elementary
Lie algebra, its restricted enveloping algebra $\fu(\epsilon)$ is a local
ring and, hence, $V_{i+1}/V_i$ always has such a 1-dimensional invariant subspace.  
Note that $V_{i+1}$ is isotropic because it is contained in $V_i^\perp$ and 
$V_i$ is isotropic. Continuing this process to step $n$ constructs an
$\epsilon$-invariant complete isotropic flag. 
\end{proof}

\begin{thm}
\label{thm:sp2n}
Let $\fg = \fsp_{2n}$ and assume that $p \neq 2$. Then
\noindent
\begin{enumerate}
\item For any elementary subalgebra $\epsilon$ of $\fg$, $\dim \epsilon \leq \frac{n(n+1)}{2}$. 
\item Any elementary subalgebra $\epsilon$  of maximal dimension $\frac{n(n+1)}{2}$ is conjugate to $\fu_{\alpha_n}$ under the adjoint action of $\Sp_{2n}$.
\item  The orbit map $\Sp_{2n} \to \Sp_{2n}\cdot \fu_{\alpha_n}$ determines a finite radicial morphism 
\ $\xymatrix@=18pt{\Sp_{2n}/P_{\alpha_n} \ar[r]& \bE(\frac{n(n+1)}{2}, \fsp_{2n})}$. 
For $p > 4n-2$, this morphism is an isomorphism.
\end{enumerate}
\end{thm}
\begin{proof}
We prove by induction that the statement of the theorem holds for a Lie algebra $\fg = \Lie G$ of any reductive group of type $C_n$.  
The statement is trivial for $n=1$. 

Assume the statement is proven for $n-1$. Let $G$  be a reductive group of type $C_n$ and let $\fg = \Lie G$. Recall that we follow the convention of \cite{Bour} for numbering of simple roots, so that the Dynkin diagram for $\fg$ looks as follows: 
\begin{equation}
\label{eq:Dynkin} 
 \xymatrix{\stackrel[1]{}{\circ} \ar@{-}[r]& \stackrel[2]{}{\circ} \ar@{-}[r] &\stackrel[3]{}{\circ} & \ldots &\stackrel[n-2]{}{\circ} \ar@{-}[r] & \stackrel[n-1]{}{\circ} & \stackrel[n]{}{\circ} \ar@{=>}[l]}
\end{equation}
Let $\fp_{\alpha_1} = \fl_{\alpha_1} \oplus \fu_{\alpha_1}$ be the maximal  
parabolic subalgebra corresponding to the simple root $\alpha_1$ with the Levi factor $\fl_{\alpha_1}$ and 
the nilpotent radical $\fu_{\alpha_1}$. To obtain the Dynkin diagram for $\fl_{\alpha_1}$ we simply remove the first node from \eqref{eq:Dynkin}. 
Hence, $\fl_{\alpha_1}$ is a reductive Lie algebra of type $C_{n-1}$, and we can apply inductive hypothesis to it.

Let $\fu_{\fl_{\alpha_1}}$ be the nilpotent radical of the standard Borel subalgebra of $\fl_{\alpha_1}$, 
and $\fu_{\fg}$ be the nilpotent radical of the  Borel subalgebra of $\fg$.   
We have a short exact sequence
\[ 
\xymatrix{0 \ar[r]& \fu_{\alpha_1} \ar[r]& \fu_{\fg} \ar[r]^{\pr}& \fu_{\fl_{\alpha_1}} \ar[r]& 0.} 
\] 

Let $\epsilon$ be an elementary subalgebra of $\fg$. Since $\epsilon$ consists of nilpotent matrices, 
it can be conjugated into the standard Borel subalgebra of $\fg$ by Lemma~\ref{lem:sympl}. Furthermore, since every element of $\epsilon$ is $p$-nilpotent, such a conjugate necessarily belongs to the nilpotent radical $\fu_{\fg}$. Hence, we may assume that $\epsilon \subset \fu_{\fg}$. By the 
induction hypothesis, 
$\dim \pr(\epsilon) \leq \frac{n(n-1)}{2}$.  Since $\fu_{\alpha_1}$ is a
Heisenberg Lie algebra of dimension $2n-1$ 
(see Example~\ref{ex:extrasp}(2)), Proposition~\ref{prop:extra-spec} implies that 
$\dim \fu_{\alpha_1} \cap \epsilon \leq n$. Hence, $\dim \epsilon \leq n + \frac{n(n-1)}{2}$.  This proves (1). 

To prove (2), we observe that the induction hypothesis 
implies that for an elementary 
subalgebra $\epsilon$ to attain the maximal dimension, 
we must have that
\[
\xymatrix@=15pt{\pr\downarrow_{\epsilon}: \epsilon \ar[r]&  \fu_{\fl_{\alpha_1}}}
\]
is surjective onto $\fu_{\fl_{\alpha_1}} \cap \fu_{\alpha_n}$, 
the nilpotent radical of the parabolic of $\fl_{\alpha_1}$
corresponding to $\alpha_n$.

Let $\{x_{\beta_i}, x_{\gamma_i}\}$ be a basis of 
$\fu_{\alpha_1}$ as defined in \eqref{eq:basis2}. 
Let $x = \sum\limits_1^{n-1} b_i x_{\beta_i} + \sum\limits_1^{n} c_i x_{\gamma_i} 
\in \fu_{\alpha_1} \cap \epsilon$. 
We want to show that $x \in \fu_{\alpha_n}$ or, 
equivalently, that $b_i = 0$ for $1 \leq i \leq n-1$.. 
Assume, to the contrary,  that $b_i \not = 0$ 
for some $i$, $1 \leq i \leq n-1$. 
Let $\mu = \gamma_{n-1} - \beta_i =  \alpha_2 + 
\ldots + \alpha_i + 2\alpha_{i+1} + \ldots + 2\alpha_{n-1} + \alpha_n$. 
Then $x_\mu \in \fu_{\fl_{\alpha_1}} \cap 
\fu_{\alpha_n}\subset \pr(\epsilon)$. Therefore, there 
exists $y = x^\prime + x_\mu \in \epsilon$ for 
some $x^\prime \in \fu_{\alpha_1}$. Note 
that $[x, x^\prime] \subset [\fu_{\alpha_1}, 
\fu_{\alpha_1}] = kx_{\gamma_n}$, and that 
$\mu + \gamma_i$ is never a root, and 
$\mu + \beta_j$ is not a root unless $j=i$. Hence,  
\[ 
[x, y] = [x, x^\prime] + [x, x_\mu]  = 
cx_{\gamma_n} + b_i[x_{\beta_i},x_\mu] = 
cx_{\gamma_n} + b_ic_{\beta_i\mu}x_{\gamma_{n-1}} \not = 0.
\]
Here, $c_{\beta_i\mu}$ is the structure constant from the equation 
$[x_{\beta_i},x_\mu] = c_{\beta_i\mu} x_{\beta_i + \mu} = 
c_{\beta_i\mu}x_{\gamma_{n-1}}$. 
This structure constant is not zero because the only 
elements in the $\beta_i$ 
string through $\mu$ are $\mu$ and $\mu+\beta_i$ 
(See \cite{Stei}, Theorem 1(d)).   
Thus, we have a contradiction with the commutativity of 
$\epsilon$. Hence, $b_i=0$ for all $i$, $1\leq i \leq n-1$, and, therefore, 
$\fu_{\alpha_1} \cap \epsilon \subset \fu_{\alpha_n}$. 
Moreover, since we assume that $\dim \epsilon$ is maximal, we must have 
$\dim \fu_{\alpha_1} \cap \epsilon = n$, and, therefore, $\fu_{\alpha_1} \cap \epsilon = \bigoplus\limits_{i=1}^n kx_{\gamma_i}$. 

Now let $x + a$ be any element in $\epsilon$ where $x \in \fu_{\alpha_1}$ and 
$a \in \fu_{\fl_{\alpha_1}} \cap \fu_{\alpha_n}$.  
We need to show that  $x \in \fu_{\alpha_n}$, that is, $x \in  \bigoplus\limits_{i=1}^n kx_{\gamma_i}$.  
Let $x = \sum b_i x_{\beta_i} + \sum c_i x_{\gamma_i}$ and assume to the contrary that $b_i \not  = 0$ for some $i$.  
Note that $[x_{\gamma_j}, \fu_{\fl_{\alpha_1}} \cap \fu_{\alpha_n}]=0$  for any $j$, $ 1 \leq j \leq n$,  
since both $x_{\gamma_j}$ and any $a \in \fu_{\fl_{\alpha_1}} \cap \fu_{\alpha_n}$
are linear combinations of root vectors  for roots that have coefficient
by $\alpha_n$ equal to 1.  
Hence, $[x+a, \gamma_{n-i}] = b_i[x_{\beta_i}, \gamma_{n-i}]  \not = 0$. 
Again, we have contradiction.  
Therefore, $\epsilon \subset \fu_{\alpha_n}$. This proves (2). 

To establish (3), we first note that $P_{\alpha_n}$ is the (reduced) stabilizer of $\fu_{\alpha_n}$ 
under the adjoint action of $\Sp_{2n}$. Arguing as in the end of the proof of Theorem~\ref{thm:sl2m}, we conclude that 
the orbit map $\Sp_{2n} \to \Sp_{2n} \cdot \fu_{\alpha_n}$ induces a finite radicial morphism 
\[
\Sp_{2n}/P_{\alpha_n} \simeq \Sp_{2n} \cdot \fu_{\alpha_n}. 
\]
Since the Coxeter number of $\Sp_{2n}$ is $2n$, the final 
statement that the above map is an isomorphism 
for $p > 4n-2$ follows from \cite[3.7]{CFP4} as 
discussed in Remark \ref{rem:sep}.
\end{proof}

In the last calculation of this section we show that any 
Grassmannian $\Grass(a,b)$ can be realized as $\bE(r, \fg)$ 
or one of the to connected components of $\bE(r,\fg)$ if we let $\fg$ be a maximal parabolic subgroup of type $A$. 

\begin{thm}
\label{prop:p1} Assume that $p>2$, and that $n \geq 4$. 
Let $m = \left\lfloor\frac{n}{2}\right\rfloor$, and 
suppose that $r \leq m$. 
Let $\fp_{r, n-r}$ be the standard maximal 
parabolic subalgebra of $\fsl_n$ corresponding to the simple root $\alpha_r$.
Then the maximal dimension of an
elementary subalgebra of $\fp_{r, n-r}$ is $\left\lfloor\frac{n^2}{4}\right\rfloor$. 
The corresponding variety of elementary 
subalgebras is homeomorphic to $\Grass(m, n-r)$ if $n$ is even 
and $\Grass(m, n-r) \sqcup \Grass(m, n-r)$ if $n$ is odd.  
\end{thm}

\begin{proof}  We consider the case of $n=2m+1$ odd. The even case is similar. 

Theorem~\ref{thm:sl2m+1} implies immediately that $\dim \epsilon 
\leq m(m+1) =\left\lfloor\frac{n^2}{4}\right\rfloor$ for any elementary 
subalgebra $\epsilon \subset \fp_{r,n-r}$. 
Since $\fu_{m,m+1}$ is a 
subalgebra of $\fp_{r, n-r}$, we have equality in the maximal case. 

To compute the variety, we first show that any elementary 
subalgebra of maximal dimension is conjugate 
to either   $\fu_{m,m+1}$ or $\fu_{m+1,m}$  under 
the adjoint action of $P_{r,n-r}$. Let $\epsilon \subset \fp_{r,n-r}$ be an 
elementary subalgebra of maximal dimension..  

By Theorem~\ref{thm:sl2m+1}, $\epsilon$ is conjugate 
to $\fu_{m,m+1}$ or $\fu_{m+1,m}$ under the adjoint 
action of $\SL_{2m+1}$. 
Assume that $\epsilon = g \fu_{m+1,m} g^{-1}$  for 
some $g \in \SL_{2m+1}$ (the case of $\fu_{m,m+1}$
 is strictly analogous).  
We proceed to show that there exists $\widehat g \in P_{r,n-r}$ 
such that $\epsilon = \widehat g \fu_{m+1,m} \widehat g^{-1}$.

Let $W(\SL_{2m+1}) \simeq N_{\SL_{2m+1}}(T)/C_{\SL_{2m+1}}(T)$ 
be the Weyl group,   $B_{2m+1}$ be the Borel subgroup of $\SL_{2m+1}$, 
and $U_{2m+1}$ be the unipotent radical of $B_{2m+1}$ .  
For an element $w \in W(\SL_{2m+1})$, we 
denote by $\wt w$ a fixed coset representative of $w$ in $N_{\SL_{2m+1}}(T)$.

Using the Bruhat decomposition, we can write $g = g_1\wt w g_2$  where 
$g_1 \in U_{2m+1}$, $g_2 \in B_{2m+1}$, 
and $w \in W(\SL_{2m+1})$. 
Since both $\fu_{m+1,m}$ and $P_{r, n-r}$ are stable under the conjugation 
by $U_{2m+1}$ and $B_{2m+1}$,   it suffices to prove the 
statement for $g = \wt w$, where $w$ is a Weyl group element. 
We make the standard identifications $W(\SL_{2m+1}) \simeq S_{2m+1}$, 
$W(L_{r, n-r}) \simeq S_r \times S_{n-r}$ 
and $W(L_{m+1,m}) \simeq S_{m+1} \times S_m$ where $L_{i,j}$ is the 
Levi factor of the standard parabolic $P_{i,j}$.

We further decompose  
$$
S_{2m+1} = W(\SL_{2m+1}) = \bigsqcup\limits_{ s \in S_r \times S_{n-r} \backslash S_{2m+1}/(S_{m+1} 
\times S_m)} (S_r \times S_{n-r})s (S_{m+1}\times S_{m})
$$ 
into double cosets. Now let $w$ be an element of $W$ 
such that $\epsilon = \wt w \fu_{m+1, m} \wt w^{-1} \in \fp_{r, n-r}$. 
To prove that $\epsilon$ is conjugate to $\fu_{m+1,m}$ by 
an element of $P_{r, n-r}$, it suffices to show that $w$ belong 
to the identity double coset. Indeed, if that is the case, then 
we can write $w = w_1 w_2$ with $w_1 \in S_r \times S_{n-r}$, 
$w_2 \in S_{m+1}\times S_{m}$. Since $w_2$ acts trivially on 
$\fu_{m+1, m}$, we get that $\epsilon = \wt w \fu_{m+1, m} 
\wt w^{-1} = \wt w_1 \fu_{m+1, m} \wt w_1^{-1}$ which is satisfactory since $\wt w_1 \in P_{r, n-r}$.

To establish that $w$ is in the identity double coset, we first prove the following claim:\\

{\it Claim.} Suppose $\wt w \fu_{m+1, m} \wt w^{-1}  
\subset \fp_{r, n-r}$. Then for any $i$, $ 1\leq i \leq r$, 
\[w^{-1}(i) \leq m+1.\] 

{\it Proof of the Claim}. We prove the claim by contradiction. Let $E_{\ell j}$ denote 
the matrix with $1$ at the entry $(\ell, j)$ and $0$ everywhere else. Suppose 
there exists $i$, $ 1 \leq i \leq r$, such that $w^{-1}(i) =j >m+1$. Since the permutation $w$ is 
bijective, and $r < m+1$,  we can find an index $\ell$ such that 
\[ \begin{cases} 
1 \leq \ell \leq m+1 \\
w(\ell) \geq r+1.
\end{cases} 
\]
The conditions on $\ell$ now imply that 
\[ \begin{cases} 
E_{\ell j} \in \fu_{m+1, m}, \\
\wt wE_{\ell j}\wt w^{-1} = E_{w(\ell)w(j)} =E_{w(\ell)i} \not \in \fp_{r, n-r},
\end{cases} 
\]
a contradiction.  This finishes the proof of the claim. 

\vspace{0.1in} The Claim implies that for any $i$, $1 \leq i \leq r$, the transposition  $(i, w^{-1}(i))$ is in $S_{m+1} \times \id \subset S_{2m+1}$. Therefore, multiplying $w$ by such transposition on the right preserves both the double coset representative of $w$ and  the property that $\epsilon = \wt w \fu_{m+1, m} \wt w^{-1}$.  Hence, $w$ is in the same double coset as a permutation which acts trivially on the first $r$ entries. But such permutations are in the identity double coset and, hence, so is $w$. This finishes the proof of the claim that $\fu_{m+1, m}$ can be conjugated to $\epsilon$ with an element in $P_{r, n-r}$.

The above discussion 
implies that 
\[
\bE(m(m+1),\fp_{r,n-r}) = P_{r,n-r}\cdot 
\fu_{m+1,m} \sqcup P_{r,n-r}\cdot \fu_{m,m+1}.
\] 
The (reduced) stabilizer of $\fu_{m+1,m}$ in $P_{r,n-r}$ is
$P_{r,m+1-r,m} \ = \ P_{m+1,m} \cap P_{r, n-r}$. Hence, the orbit map $\xymatrix@=12pt{P_{r, n-r} \ar[r]& P_{r, n-r}\cdot 
\fu_{m+1,m}}$ induces a homeomorphism 
\[\xymatrix@=15pt{\Grass(m, n-r) \cong P_{r, n-r}/ P_{r, m+1-r, m} \ar[r]^-\sim& P_{r, n-r}\cdot 
\fu_{m+1,m},}
\] 
and similarly for the other component. 
\end{proof}

Theorem~\ref{prop:p1} has the following immediate corollary.
\begin{cor}
\label{cor:g1n} Let $\fg_{1,2m} \subset \gl_{2m+1}$ be as defined in Example~\ref{ex:semi}(1).
The maximal dimension of an
elementary subalgebra of $\fg_{1,2m}$ is $m(m+1)$.  For $m \geq 2$,  $\bE({m(m+1)}, \fg_{1,2m})$ is homeomorphic to 
$\Grass(m, 2m)\sqcup \Grass(m-1, 2m)$. 
\end{cor}


\section{Radicals, socles, and geometric invariants for $\fu(\fg)$-modules}
\label{rad-soc}

As throughout this paper, $\fg$ denotes a finite dimensional $p$-restricted Lie algebra over
$k$.  We recall that $\fg$ is the Lie algebra $\Lie(\ul \fg)$ of a uniquely 
defined infinitesimal group scheme $\ul \fg$ of height 1 (see, for example, \cite{DG}).
In \cite{SFB2}, {\it a rank variety} $V(G)_M$ was constructed for any finite dimensional
representation $M$ of the infinitesimal group scheme $G$.  The variety $V(G)_M$ is a 
closed subset of $V(G)$, the variety of (infinitesimal) 1-parameter subgroups of
$G$.  As shown in \cite{SFB2}, these rank varieties can be identified with  cohomological
support varieties defined in terms of the action of $\HHH^*(G,k)$ on $\Ext_{G}^*(M,M)$.

For infinitesimal group schemes $G$ of height 1 (i.e., of the form $\ul \fg$ for some 
finite dimensional $p$-restricted Lie algebra), we consider more complete invariants of
representations of $G$ which one can think of as more sophisticated variants
of ``higher rank varieties."  Our investigations
follow that of our earlier paper \cite{CFP2} in which we considered representations
of elementary abelian $p$-groups.  Because the group algebra $k(\bZ/p^{\times r})$
is isomorphic to the restricted enveloping algebra $\fu(\fg_a^{\oplus r})$ 
of the Lie algebra $\fg_a^{\oplus r}$
(commutative, with trivial $p$-restriction), that investigation is in fact a very special
case of what follows.  

We use our earlier work for elementary abelian $p$-groups as a guide for the study
of $\fu(\fg)$-modules for an arbitrary $\fg$.  In
particular, rather than considering isomorphism types of a given module upon restriction 
to elementary subalgebras of a given rank $r$, we consider dimensions of the 
radicals (respectively, socles) of such restrictions.  A key result is Theorem \ref{upper-lower}
which verifies that these dimensions are lower (resp., upper) semi-continuous.
As seen in Theorem~\ref{thm:radvar}, this implies that the non-maximal
radical and socle varieties associated to a $\fu (\fg)$-module $M$ are closed.

\vskip .1in

The following is a natural extension of the usual 
support variety in the case $r=1$ (see \cite{FPar2})
and of the variety $\Grass(r, V)_M$ of \cite[1.4]{CFP2} 
for $\fg = \fg_a^{\oplus n}$.
If $\epsilon \subset \fg$ is an elementary subalgebra 
and $M$ a $\fu(\fg)$-module, then
we shall denote by $\epsilon^*M$ the restriction of 
$M$ to $\fu(\epsilon) \subset \fu(\fg)$.

\begin{defn}  
\label{def:support} 
For any $\fu(\fg)$-module $M$ and any positive integer $r$, we define
$$\bE(r,\fg)_M \ = \ \{ \epsilon \in \bE(r,\fg); \epsilon^*M \text{~ is not projective} \}.$$
In particular,
$$
\bE(1,\fg)_M \ = \     \Proj k[V(\ul \fg)_M] \ \subset \  \Proj k[V(\ul \fg)] \ = \ \bE(1,\fg)
$$ 
is the projectivization of the closed subvariety of $V(\ul \fg) = \cN_p(\fg)$
consisting of those one dimensional Lie subalgebras (with trivial $p$-restriction)
restricted to which $M$ is not projective. 
\end{defn}

The following proposition tells us that the geometric 
invariant $M \mapsto \bE(r,\fg)_M$
can be computed in terms of the more familiar (projectivized)
 support variety $\bE(1,\fg)_M \ = \ \Proj(V(\ul \fg)_M)$.  
 
\begin{prop}
\label{reduce}
For any $\fu(\fg)$-module $M$ and positive integer $r$,
\begin{equation}
\label{reduceto1}
\bE(r,\fg)_M \ = \ \{ \epsilon \in \bE(r,\fg); \ \epsilon 
\cap V(\ul \fg)_M \not= 0 \}
\end{equation}
where the intersection $\epsilon \cap V(\ul \fg)_M $ 
is as subvarieties of $\fg$.
\end{prop}

\begin{proof}
By definition, $\epsilon \in \bE(r,\fg)_M$ if and only 
if $\epsilon^*M$ is not free which is
the case if and only if $V(\ul \epsilon)_{\epsilon^*M} \not= 0$.  
Since $\epsilon \subset \fg$ induces an isomorphism
$$
\xymatrix{V(\ul \epsilon)_{\epsilon^*(M)} \ar[r]^-{\sim}& \ 
V(\ul \epsilon) \cap V(\ul \fg)_M}
$$ (see \cite{FPar2}), this 
is equivalent to $\epsilon \cap V(\ul g)_M \not= 0$.
\end{proof}

\begin{prop}
\label{closed}
For any $\fu(\fg)$-module $M$ and  for any $r \geq 1$,
$$\bE(r,\fg)_M \ \subset \ \bE(r,\fg)$$
is a closed subvariety.

Moreover, if $G$ is an algebraic group with $\fg = \Lie(G)$ 
and if $M$ is a rational $G$-module,
then $\bE(r,\fg)_M \ \subset \ \bE(r,\fg)$ is $G$-stable.
\end{prop}

\begin{proof}
Let $\Proj \epsilon \subset \bE(1, \fg)$ be the projectivization of the linear subvariety $\epsilon \subset \fg$.  
Let $X_M = \{\epsilon \in \Grass(r, \fg) \, | \, \Proj \epsilon \cap  \bE(1,\fg)_M \not = \emptyset\}$. 
Then $X_M \subset \Grass(r, \fg)$ is a 
closed subvariety (see \cite[ex. 6.14]{Har10}). Since $\bE(r,\fg)_M = \bE(r,\fg) \cap X_M$ by 
Prop.~\ref{reduce}, we conclude that $\bE(r,\fg)_M$ is a closed subvariety of $\bE(r,\fg)$.  

For $\fg = \Lie(G)$, $M$ a rational $G$-module, and $x \in G$, 
denote by $M^x$ the module $M$ twisted by $x$. For $\epsilon \in \bE(r, \fg)$, 
denote by $\epsilon^x$ the image of $\epsilon$ under the adjoint action of $x$ on $\bE(r,\fg)$.
The adjoint action by $x^{-1}$ 
induces an isomorphism $\xymatrix{\alpha_{x^{-1}}: 
\fu(\epsilon^x) \ar[r]^-{\sim}& \fu(\epsilon)}$, 
and the pull-back of $M$ along this 
isomorphism equals  $(\epsilon^x)^*(M^x)$. 
Since $M \simeq M^x$ as $\fu(\fg)$-modules, we conclude that 
$\bE(r,\fg)_M$ is $G$-stable.
\end{proof}

Proposition \ref{reduce}  implies the following result concerning the realization of
subsets of $\bE(r,\fg)$ as subsets of the form $X = \bE(r,\fg)_M$.  We remind
the reader of the definition of the  module $L_\zeta$ associated to a 
cohomology class $\zeta \in \HHH^n(\fu(\fg),k)$:  $L_\zeta$ is  the 
kernel of the map $\zeta: \Omega^n(k) \to k$ determined by $\zeta$, where $\Omega^n(k)$
is the $n^{th}$ Heller shift of the trivial module $k$ 
(see \cite{Ben} or Example~\ref{ex:heller}).

\begin{cor}
\label{realize}
A subset $X \subset \bE(r,\fg)$ has the form $X = \bE(r,\fg)_M$
for some $\fu(\fg)$-module $M$ if and only if there exists
a closed subset $Z \subset \bE(1,\fg)$ such that
\begin{equation}
\label{form}
X \ = \ \{ \epsilon \in \bE(r,\fg); \ \Proj \epsilon \cap Z \not= \emptyset \}.
\end{equation}
Moreover, such an $M$ can be chosen to be a tensor product of
 modules $L_\zeta$ with each $\zeta$ of even cohomological degree.
\end{cor}

\begin{proof}
We recall that any closed, conical subvariety of $V(\ul \fg)$ (i.e., any closed
subvariety of $\bE(1,\fg)$) can be realized as the (affine) support of a tensor product of 
 modules $L_\zeta$ (see \cite{FPar2}) and that the support of any
finite dimensional $\fu(\fg)$-module is a closed, conical subvariety of $V(\ul \fg)$.
Thus, the proposition follows immediately from Proposition \ref{reduce}.
\end{proof}

\begin{ex}
As one specific example of Corollary~\ref{realize}, we take some even degree
cohomology class $0 \not= \zeta \in \HHH^{2m}(\fu(\fg),k)$ and $M = L_\zeta$.   
We identify $V(\ul \fg)$ with the
spectrum of $\HHH^{\rm ev}(\fu(\fg),k)$ (for $p>2$), 
so that $\zeta$ is a (homogeneous) algebraic function 
on $V(\ul\fg)$.
Thus $V(\ul\fg)_{L_\zeta} = Z(\zeta) \subset V(\ul\fg)$ (see \cite[Theorem 7.5]{SFB2}), 
the zero locus of the function $\zeta$.
Then, 
$$\bE(r,\fg)_{L_\zeta} \ = \ \{ \epsilon \in \bE(r,\fg); \ \epsilon \cap Z(\zeta) \not= \{ 0 \} \}.$$

On the other hand, if $\zeta \in \HHH^{2m+1}(\fu(\fg),k)$ 
has odd degree and $p > 2$, then $V(\ul\fg)_{L_\zeta}
= V(\ul\fg)$, so that $\bE(r,\fg)_{L_\zeta} \ = \ \bE(r,\fg)$.
\end{ex}

\begin{remark}
As pointed out in \cite[1.10]{CFP2} in the special case $\fg = \fg_a^{\oplus 3}$ and
$r =2$, not every closed subset $X \subset \bE(r,\fg)$ 
has the form (\ref{form}).
\end{remark}

\begin{ex}
\label{npv}
We consider another computation of $\bE(r,\fg)_M$.  Let $G$ be a reductive group
and assume that $p$ is good for $G$.  Let $\lambda$ be a dominant weight and consider
the induced module $M = \HHH^0(\lambda) = \Ind_B^G \lambda$.  By a result of Nakano, 
Parshall, and Vella \cite[6.2.1]{NPV}, 
$V(\ul\fg)_{\HHH^0(\lambda)} = G \cdot \fu_J$, where $\fu_J$ is 
the nilpotent radical of a suitably chosen parabolic subgroup $P_J \subset G$.
Then, 
$$\bE(r,\fg)_{\HHH^0(\lambda)} \ = \ G \cdot \{ \epsilon \in \bE(r,\fg);
\ \epsilon \cap \fu_J \not= \{ 0 \} \}.$$
\end{ex}
\vspace{0.3in}

We now proceed to consider invariants of $\fu(\fg)$-modules associated to 
$\bE(r,\fg)$ which for $r > 1$ are not determined 
by the case $r = 1$.   As before, for a 
given $M$ and a given $r \geq 1$, we consider the 
restrictions $\epsilon^*(M)$ for $\epsilon \in \bE(r,\fg)$.

\begin{defn}
Let $\fg$ be a $p$-restricted 
Lie algebra and $M$ a finite dimensional
$\fu(\fg)$-module.  For any $r \geq 1$, any 
$\epsilon \in \bE(r,\fg)$, and any $j, 1 \leq j \leq (p-1)r$, 
we consider
$$\Rad^j(\epsilon^*(M)) \ = \ \sum_{j_1 + 
\cdots +j_r=j}\Im\{ u_1^{j_1}\cdots u_r^{j_r}: M \to M \}$$
and 
$$\Soc^j(\epsilon^*(M)) \ = \ \bigcap_{j_1 + 
\cdots +j_r=j}\Ker\{ u_1^{j_1}\cdots u_r^{j_r}: M \to M\},$$
where $\{ u_1,\ldots,u_r\}$ is a basis for $\epsilon$.

For each $r\geq 1$ and each $j, 1 \leq j \leq (p-1)r$, 
we define the local $(r,j)$-radical rank of $M$
and the local $(r,j)$-socle rank of $M$ to be the 
(non-negative) integer valued functions
$$\epsilon \in \bE(r,\fg)  \ \mapsto \ \dim \Rad^j(\epsilon^*(M))$$
and 
$$\epsilon \in \bE(r,\fg)  \ \mapsto \ \dim \Soc^j(\epsilon^*(M))$$
respectively.
\end{defn}

\begin{remark}
\label{rem:perp}
If $M$ is a $\fu(\fg)$-module, we denote by $M^\# = \Hom_k(M,k)$ the dual of $M$ whose
$\fu(\fg)$-module structure arises from that on $M$ using the antipode of $\fu(\fg)$.   Thus, 
if $X \in \fg$ and $f \in M^\#$, then $(X\circ f)(m) = - f(X\circ m)$.
 If $i: L \subset M$ is a $\fu(\fg)$-submodule, then we denote by $L^\perp \subset M^\#$ the
submodule defined as the kernel of $i^\#: M^\# \to L^\#$.  We remind the reader that
\begin{equation}
\label{perp}
\Soc^j(\epsilon^*(M^\#)) \ \simeq \ (\Rad^j(\epsilon^*M))^\perp
\end{equation}
(as shown in \cite[2.2]{CFP2}).
\end{remark}

The following elementary observation  enables
us to conclude in \cite{CFP4} that the constructions of \S 4 determine
vector bundles on $G$-orbits of $\bE(r,\Lie G)$.

\begin{prop}
\label{prop:orbit}
If $\fg = \Lie(G)$ and $M$ is a rational $G$-module, 
then the local $(r,j)$-radical rank of $M$ and
the local $(r,j)$-socle rank of $M$ are constant 
on $G$-orbits of $\bE(r,\fg)$.
\end{prop}

\begin{proof}
Let $g \in G$, and let $\epsilon \in \bE(r,\fg)$. 
We denote by $\epsilon^g \in \bE(r,\fg)$ 
the image of $\epsilon$ under the adjoint action of
$G$ on $\bE(r,\fg)$, and let $g \cdot (-): M \to M$ 
be the action of $G$ on $M$.   
Observe that
\[ 
\xymatrix{g: M\ar[rr]^-{m \mapsto gm}&& M^g}\] 
defines an isomorphism of rational $G$-modules, 
where the action of $x \in G$ on
$m \in M^g$ is given by the action of $ gxg^{-1}$ 
on $m$ (with respect to the $G$-module
structure on $M$).   Thus,  the 
proposition follows from the observation that the 
pull-back of $\epsilon^{g*}(M^g)$ equals
$\epsilon^*(M)$ under the isomorphism given by conjugation by $g$:
$\xymatrix{\fu(\epsilon) \ar[r]^-{\sim}& \fu(\epsilon^g)}$. 
\end{proof}

The following discussion leads to Theorem~\ref{upper-lower} 
which establishes the lower and upper semi-continuity
of local $(r,j)$-radical rank and local $(r,j)$-socle rank respectively.

\begin{notation}
\label{note2} We fix a basis $\{x_1, \ldots, x_n\}$ 
of $\fg$ and use it to identify $M_{n,r} \simeq \fg^{\oplus r}$ 
(as vector spaces). Let $\Sigma \subset \{1, \ldots, n\}$ 
be an $r$-subset. Recall the section 
$s_\Sigma: U_\Sigma \to \bM_{n,r}^\circ$ of (\ref{section}) 
that sends an $r$-plane $\epsilon \in U_\Sigma$ 
to the $n \times r$ matrix $A^\Sigma_\epsilon$ with the 
$r \times r$ submatrix corresponding to $\Sigma$ 
being the identity and the columns generating the plane 
$\epsilon$.  Extend the map $s_\Sigma$ to 
$s_\Sigma: U_\Sigma \to \bM_{n,r}$ and consider the 
induced map on coordinate algebras:
 
\begin{equation}
\label{quot}
\xymatrix{k[\bM_{n,r}] = k[T_{i,s}] \ar[r]^-{s_\Sigma^*}& 
k[U_\Sigma}]
  \end{equation}
  We define 
 \[T_{i,s}^\Sigma \equiv s_\Sigma^*(T_{i,s})\]
It follows from the definition that 
$T_{i,s}^\Sigma = \delta_{\alpha^{-1}(i),s}$ for $i \in  \Sigma$,
where $\alpha: \{1, \dots, r\} \to \Sigma$ is the function with 
$\alpha(1) < \dots < \alpha(r)$, 
and that 
$T_{i,s}^\Sigma$ for $i \notin \Sigma$
are algebraically independent generators of $k[U_\Sigma]$. 

Let $V_\Sigma \ \equiv \ \bE(r,\fg) \cap U_\Sigma$. We 
define the set $\{Y^\Sigma_{i,s}\}$ of algebraic 
generators of $k[V_\Sigma]$ as images of $\{T^\Sigma_{i,s}\}$ 
under the map of coordinate algebras
 induced by the closed immersion $V_\Sigma \subset U_\Sigma$:
$$
\xymatrix{k[U_\Sigma] \ \ar@{->>}[r] & k[V_\Sigma]}, 
\quad T_{i,s}^\Sigma \mapsto Y_{i,s}^\Sigma
$$
It again follows that
$Y^\Sigma_{i,s} = \delta_{\alpha^{-1}(i),s}$, for 
$i \in \Sigma$ and $\alpha$ as above.  
 For each $\epsilon \in V_\Sigma\subset U_\Sigma$ 
(implicitly assumed to be a $k$-rational point),  we have
$$
Y_{i,s}^\Sigma(\epsilon) = T^\Sigma_{i,s}(\epsilon) = 
s_\Sigma^*(T^\Sigma_{i,s})(\epsilon) =T_{i,s}(s_\Sigma(\epsilon)).
$$ 
 Hence,  
\begin{equation}
\label{eq:matrix} 
A^\Sigma_\epsilon= [Y_{i,s}^\Sigma(\epsilon)].
\end{equation} \end{notation}

\begin{defn}
\label{defn:Theta} For a $\fu(\fg)$-module $M$,
and for a given $s, 1 \leq s \leq r$, we define the endomorphism 
 of $k[V_\Sigma] $--modules
 \begin{equation}
 \label{theta-s}
 \Theta_s^\Sigma \equiv \sum_{i=1}^n x_i  \otimes Y_{i,s}^\Sigma: M \otimes k[V_\Sigma]  
 \to M \otimes k[V_\Sigma], 
 \end{equation} via
 $$
  m \otimes 1 \mapsto  \sum_i x_im \otimes  Y^\Sigma_{i,s}. 
  $$
\end{defn}

We refer the reader to \cite[III.12]{Har} for the 
definition of an upper/lower semi-continuous function on a topological space. 
\begin{thm}
\label{upper-lower}
Let $M$ be a $\fu(\fg)$-module, $r$ a positive integer, and $j$ an integer satisfying 
$1 \leq j \leq (p-1)r$.  Then  the local $(r,j)$-radical 
rank of $M$ is a lower semicontinuous function and 
the local $(r,j)$-socle rank of $M$ is an upper
semicontinuous function on $\bE(r,\fg)$.
\end{thm}

\begin{proof}  It suffices to show that the local $(r,j)$-radical rank of $M$ is lower semi-continuous 
when restricted along each of the open immersions $V_\Sigma \subset \bE(r,\fg)$.
For $\epsilon \in V_\Sigma$ with residue field $K$, the specialization 
of $\Theta_s^\Sigma$ at $\epsilon$ defines a linear operator
$\Theta^\Sigma_s(\epsilon) =  \sum_{i=1}^nY_{i,s}^\Sigma(\epsilon) x_i$ on $M_K$: 
$$
m \mapsto \Theta_s^\Sigma(\epsilon) \cdot m = 
\sum_{i=1}^nY_{i,s}^\Sigma(\epsilon) x_im.
$$
Since the columns of $[Y^\Sigma_{i,s}(\epsilon)]$ generate $\epsilon$ by \eqref{eq:matrix}, we get that 
\begin{equation}
\label{rad-ep}
\Rad(\epsilon^*M) \ = \ \sum_{s=1}^r \Im\{\Theta_s^\Sigma(\epsilon)  : M_K \to M_K \}
\end{equation}
and 
\begin{equation}
\label{rad-j}
\Rad^j(\epsilon^*M) \ = \ \sum\limits_{j_1 + \cdots + j_r = j}
\Im\{ \Theta_1^\Sigma(\epsilon)^{j_1}\ldots   \Theta_r^\Sigma(\epsilon)^{j_r}: M_K \to M_K \} = 
\end{equation}
\[
\Im\{\bigoplus\limits_{j_1 + \cdots + j_r = j} \Theta_1^\Sigma(\epsilon)^{j_1}\ldots   
\Theta_r^\Sigma(\epsilon)^{j_r}:M_K^{\oplus r(j)} \to M_K\}
\]
where $r(j)$ is the number of ways to write $j$ as the sum of non-negative integers $j_1 + \cdots +j_r$.
Hence, the usual argument for lower semicontinuity of the dimension of images of a homomorphism
of finitely generated free modules applied to the $k[V_\Sigma]$-linear map 
\[\bigoplus_{j_1 + \cdots + j_r = j} (\Theta_1^\Sigma)^{j_1}\ldots   (\Theta_r^\Sigma)^{j_r}:
(M \otimes k[V_\Sigma] )^{\oplus r(j)} \to M \otimes k[V_\Sigma]. 
\]
enables us to conclude  that the function
\begin{equation}
\label{lowerc}
\epsilon \in \bE(r,\fg) \ \mapsto \ \dim\Rad^j(\epsilon^*M) \quad \text{ is lower semi-continuous}.
\end{equation}

The upper semi-continuity of socle ranks now follows by Remark~\ref{rem:perp}.
 \end{proof}

\begin{remark}
\label{concrete}
To get some understanding of the operators $\Theta_s^\Sigma(\epsilon)$ occurring in the
proof of Theorem \ref{upper-lower}, we work out the very special case in which $\fg = \fg_a \oplus \fg_a$,
$r = 1$ (so that $\bE(r,\fg) = \bP^1$), and $j=1$.  We fix a basis $\{ x_1, x_2 \}$  for $\fg$ 
which induces the identification $\fg \simeq \bA^2$. 
The two possibilities for $\Sigma \subset \{ 1, 2 \}$ are $\{ 1 \}, \{ 2 \}$.
Let $k[T_1,T_2]$ be the coordinate ring for $\bA^2$ (corresponding to the fixed basis $\{x_1, x_2\}$.   Let $\Sigma = \{1\}$. 
We have  $V_{\{1\}} = U_{\{1\}} = \{[a:b] \, | \, a \not = 0 \} \simeq \bA^1$ and the section $s_{\{1\}}: V_{\{1\}} \to \bA^2$ 
given explicitly as $[a:b] \mapsto (1,b/a)$. 
The corresponding map of coordinate algebras as in \eqref{quot} is given by 
\[k[\bA^2] = k[T_1, T_2] \to k[V_{\{1\}}]\simeq k[\bA^1]\]
$$ T_1 \mapsto 1, T_2 \mapsto s_{\{1\}}^*(T_2)$$

Then for a $\fu(\fg)$-module $M$, $\epsilon = \langle a,b \rangle \in \bP^1$ with $a \not= 0$, and $m \in M$, we have
\begin{equation}
\label{Theta1}
\Theta^{ \{ 1 \}} \ = \  x_1  \otimes 1 + x_2 \otimes s_{\{1\}}^*(T_2) : M \otimes k[V_{\{1\}}] \to M \otimes k[V_{\{1\}}];
\end{equation}
$$\Theta^{ \{ 1 \} }(\epsilon) \ = \ x_1 + \frac{b}{a} x_2, \quad m \mapsto x_1(m) + \frac{b}{a}x_2(m).$$

\end{remark}

We extend the formulation of ``generalized support varieties"  introduced in  \cite{FP4}
for $r=1$ and in \cite{CFP2} for elementary abelian $p$-groups (or, equivalently, for
$\fg = \fg_a^{\oplus r}$) to any $r$ and an arbitrary $p$-restricted Lie algebra $\fg$.  

\begin{defn}
\label{def:radvar} For any finite dimensional $\fu(\fg)$-module $M$,  any positive integer $r$, and any 
$j$, $1 \leq j \leq (p-1)r$, we define
\[\bRad^j(r,\fg)_M \ \equiv \  \{ \epsilon \in \bE(r,\fg):  \dim(\Rad^j(\epsilon^*M)) 
< \max\limits_{\epsilon^\prime \in \bE(r,\fg)} \dim\Rad^j(\epsilon^{\prime*}M) \}
\]
\[ \bSoc^j(r,\fg)_M \ \equiv \  \{ \epsilon \in \bE(r,\fg):  \dim(\Soc^j(\epsilon^*M)) 
> \min \limits_{\epsilon^\prime \in \bE(r,\fg)} \dim\Soc^j(\epsilon^{\prime*}M) \}
\]
\end{defn}

These notions are somewhat similar to the support varieties. For example, we have
the following.

\begin{lemma} \label{lem:radsoc}
Suppose that $\bE(r,\fg)_M \neq \bE(r,\fg)$. Then $\bRad^1(r,\fg)_M \ \simeq \ \bE(r,\fg)_M \
\simeq \ \bSoc^1(r,\fg)_M$.
\end{lemma}

\begin{proof}
The hypothesis implies that there exists an elementary subalgebra $\epsilon$ such that 
$\epsilon^* M$ is a free $\fu(\epsilon)$-module. Let $n = \Dim(M)/p^r$. Then we have an
isomophism of $\fu(\epsilon)$-modules, $\fu(\epsilon)^n \simeq \epsilon^* M$. If on the
other hand,  $\fc$ is an elementary subalgebra that does not act freely on 
$M$, then any homomoprhism $\xymatrix{\fu(\fc)^n \ar[r] & \fc^* M}$
must fail to be surjective, as otherwise it would be an isomorphism. It follows that 
the dimension of $\fc^* M/\Rad(\fc^* M)$ is larger 
than $n$, by Nakayama's Lemma. So $\Dim(\Rad(\fc^* M))$ is less than 
$\Dim(\Rad(\epsilon^* M))$. The proof that $\bSoc^j(r,\fg)_M = \bE(r,\fg)_M$ 
is a dual argument. 
\end{proof}

\begin{thm}
\label{thm:radvar} 
Let $M$ be a finite dimensional $\fg$-module, and let $r,j$ be positive integers 
such that $1 \leq j\leq (p-1)r$. Then 
$\bRad^j(r,\fg)_M$, $\bSoc^j(r,\fg)_M$ are proper closed subvarieties in $\bE(r,\fg)$.  
\end{thm} 
\begin{proof} Follows immediately from Theorem~\ref{upper-lower}.
\end{proof}

One approach to our first application, requires the following elementary fact.
\begin{lemma}
\label{le:comb}
Let $k[x_1, \ldots, x_n]$ be a polynomial ring, let $x_1^{i_1}\ldots x_n^{i_n}$ be a
monomial of degree $i$ and assume that $p=char \, k >i$.
There exist linear polynomials without constant term  $\lambda_0, \ldots, \lambda_m$
on the variables $x_1, \ldots, x_n$, and scalars $a_0, \ldots, a_m \in k$ such that
\[
x_1^{i_1}\ldots x_n^{i_n}  = a_0\lambda_0^i + \ldots + a_m\lambda_m^i.
\]
\end{lemma}

\begin{proof}
It suffices to prove the statement for  $n=2$, thanks to an easy induction argument (with respect to $n$).
Hence, we assume that we have only two variables, $x$ and $y$.

Let $\lambda_j = jx+y$ for $j=0, \ldots, i$, so that we have $i+1$ equalities
of the form $(jx+y)^i =\lambda_j^i$ for $j = 0, \dots, i$.
Treating monomials on $x,y$ as variables, we interpret this as a system of $i+1$
 equations on $i+1$ variables with the matrix
\[
\begin{pmatrix}0&0&\ldots& 0 &\ldots& 0&1& \\
1& i &  \ldots & i \choose j & \ldots & i& 1&   \\
2^i& 2^{i-1}i & \ldots & 2^{i-j} {i \choose j} & \ldots &2i& 1&   \\
\vdots&\vdots&\ddots&\vdots&\ddots&\vdots&\vdots& \\
i^i& i^{i-1}i &  \ldots & i^{i-j} {i \choose j} & \ldots &i^2& 1&
\end{pmatrix}
\]
By canceling the  coefficient ${i \choose j}$ 
in the $(j+1)$-st column (which is
 non-trivial since $p>i$)
we reduce the determinant of this matrix to 
a non-trivial Vandermonde determinan
t. Hence, the matrix is invertible.
We conclude the monomials $x^jy^{i-j}$ can be 
expressed as linear combinations of the free terms
$\lambda_0^i, \ldots, \lambda_i^i$.
\end{proof}

Determination of the closed subvarieties 
$\bRad^j(r,\fg)_M , \ \bSoc^j(r,\fg)_M$ of
$\bE(r,\fg)$ appears to be highly non-trivial.  The 
reader will find a few computer-aided calculations
in \cite{CFP2} for $\fg = \fg_a^{\oplus n}$.  The following 
proposition presents some information for $\bE({n-1}, \gl_n)$.

\begin{prop}
\label{open_orbit}
Assume that $p \geq n$.
Let $X \in \gl_n$ be a regular nilpotent element, and let 
$\epsilon \in \bE({n-1}, \gl_n)$ be an $n-1$-plane with
 basis $\{X, X^2, \ldots, X^{n-1}\}$. Then 
$\GL_n \cdot \epsilon$ is an open $\GL_n$-orbit for $\bE({n-1},\gl_n)$.
\end{prop}

\begin{proof}
Let $V$ be the defining $n$-dimensional representation of $\gl_n$.
Let $\epsilon^\prime$ be any elementary Lie subalgebra of
$\gl_n$ of dimension $n-1$.
If $\epsilon^\prime$ contains a regular nilpotent element $Y$, 
then $\epsilon^\prime$ has basis
$\{Y, Y^2, \ldots, Y^{n-1}\}$, since the centralizer of a 
regular nilpotent element in $\gl_n$ is generated
as a linear space by the powers of that nilpotent element.   
Hence, in this case $\epsilon^\prime$ is
conjugate to the fixed plane $\epsilon$.  Moreover, 
$\Rad^{n-1}(\epsilon^{\prime*}V) = \Im\{Y^{n-1}: V \to V\}$,
and, hence, $\dim \Rad^{n-1}(\epsilon^{\prime*}V) =1$.

Suppose $\epsilon^\prime$ does not contain a regular 
nilpotent element. Then for any matrix
$Y \in \epsilon^\prime$, we have $Y^{n-1} =0$.
Lemma~\ref{le:comb}  implies that any monomial of degree $n-1$ 
on elements of $\epsilon^\prime$ is trivial.
Therefore, $\Rad^{n-1}(\epsilon^{\prime*}V)=0$. We conclude 
that $\GL_n\cdot \epsilon$ is the complement
to $\bRad^{n-1}(n-1, \gl_n)_V$ in $\bE({n-1}, \gl_n)$.  
Theorem~\ref{upper-lower} now implies that
$\GL_n\cdot\epsilon$ is open.
\end{proof}

\begin{ex}
\label{ex:gl3} In this example we describe the geometry  of $\bE(2, \gl_3)$ making an 
extensive use of the $\GL_3$-action. Further calculations involving more geometry 
are currently being investigated.

Assume that $p>3$.
Fix a regular nilpotent element $X \in \gl_3$. Let 
$\epsilon_1 = \langle X, X^2\rangle$ be the 
2-plane in $\gl_3$ with the basis $X, X^2$, and let 
\[
C_1 = \GL_3 \cdot \,\epsilon_1 \subset \bE(2, \gl_3)
\]
be the orbit of $\epsilon_1$ in  $\bE(2, \gl_3)$.
By Proposition~\ref{prop:sub} or by Proposition \ref{open_orbit}
this is an open subset of $\bE(2, \gl_3)$.
Since $\bE(2, \gl_3)$ is irreducible (see Example~\ref{ex:premet}), 
$C_1$ is dense. We have $\dim C_1 = \dim \overline{C_1} = \dim \bE(2, \gl_3) =4$.

The closure of $C_1$ contains two more (closed) $\GL_3$ stable subvarieties, each one of dimension  
$2$. They are  the $\GL_3$ saturations in $\bE(2, \gl_3)$
of the elementary subalgebras $\fu_{1,2}$ (spanned by $E_{1,2}$ and 
$E_{1,3}$), and $\fu_{2,1}$ (spanned by $E_{1,3}$ and 
$E_{2,3}$). Since the stabilizer of $\fu_{1,2}$ (resp. $\fu_{2,1}$) is the standard parabolic $P_{1,2}$ (resp. $P_{2,1}$), 
the corresponding orbit is readily identified with $\GL_3/P_{1,2} \simeq 
\Grass(2,3) = \bP^2$ (resp., $\GL_3/P_{2,1} \simeq  \bP^2$) (see Remark~\ref{rem:sep}).
\end{ex}

\begin{prop} \label{prop:open} Let $\fu$ be a $p$-restricted Lie algebra with trivial $p$-restriction map.
Then the locus of elementary subalgebras $\epsilon \in \bE(r,\fu)$ such that $\epsilon$ is maximal (that is, 
not properly contained in any other elementary  subalgebra of $\fu$) is an open subset of $\bE(r,\fu)$. 
\end{prop}
\begin{proof} If no maximal elementary subalgebras are contained in $\bE(r,\fu)$, then the statement is clear. 
Hence, we may assume that there is at least one maximal elementary subalgebra $\epsilon \in \bE(r, \fu)$. 

Regard $\fu$ as acting on itself by the adjoint representation. 
Note that we necessarily have $\epsilon \subset \Soc(\epsilon^*(\fu_{\rm ad}))$. Moreover,  our hypothesis that $x^{[p]}=0$ 
for any $x \in \fu$ implies that this inclusion is an 
equality if and only if $\epsilon$ is a maximal elementary subalgebra. Hence,
\[\dim\Soc(\epsilon^*(\fu_{\rm ad})) \geq \dim \epsilon =r\]
with equality if and only if $\epsilon$ is maximal. We conclude that the locus of elementary subalgebras 
$\epsilon \in \bE(r,\fu)$ such that $\epsilon$ is nonmaximal equals the nonminimal 
socle variety $\bSoc^1(r, \fu)_{\fu_{\rm ad}}$. The statement now follows from 
Theorem~\ref{thm:radvar}.
\end{proof}


\section{Modules of constant  $(r,j)$-radical rank and/or constant $(r,j)$-socle rank}
\label{sec:constant}

In previous work with coauthors, we have considered the interesting class of modules
of constant Jordan type (see, for example, \cite{CFP1}).  In the terminology of this paper, these
are $\fu(\fg)$-modules $M$ with the property that the isomorphism type 
of $\epsilon^*M$ is independent of $\epsilon \in \bE(1, \fg)$.  In the special case $\fg = 
\fg_a^{\oplus n}$, further classes of special modules were considered by replacing this
condition on the isomorphism type of  $\epsilon^*M$ for $\epsilon \in \bE(1, \fg_a^{\oplus n})$
by the ``radical" or ``socle" type of $\epsilon^*M$ for $\epsilon \in \bE(r,\fg_a^{\oplus n})$.  

In this section, we consider $\fu(\fg)$-modules of constant $(r,j)$-radical rank and
constant $r$-radical type (and similarly for socles).
As already seen in \cite{CFP2} in the special case $\fg = \fg_a^{\oplus n}$, 
the variation of radical and socle behavior
for $r > 1$ can be quite different.  Moreover, having constant $r$-radical type
does not imply the constant behavior for a different $r$.

As we investigate in \cite{CFP4}, a $\fu(\fg)$-module of constant $(r,j)$-radical rank or
constant $(r,j)$-socle rank determines a vector bundle on $\bE(r,\fg)$, thereby providing good
motivation for studying such modules.  While a great many examples of such $\fu(\fg)$-modules, 
some well known, can be constructed from rational $G$-modules, there are 
numerous others which do not arise in this way. Some examples are given in 
\ref{ex:nullext}, \ref{prop:zeta} and \ref{prop:existzeta}. 
Although identifying the associated vector bundles is hard, some such vector bundles might
prove to be of geometric importance.

\begin{defn}
\label{defn:constant}  
Fix integers $r > 0$ and $ j, 1 \leq j < (p-1)r$.  A $\fu(\fg)$-module $M$
 is said to have constant $(r,j)$-radical rank (respectively, $(r,j)$-socle rank)
 if the dimension of $\Rad^j(\epsilon^*M)$ (respectively, $\Soc^j(\epsilon^*M)$)
is independent of $\epsilon \in \bE(r,\fg)$.

We say that $M$ has constant $r$-radical type (respectively, $r$-socle type) if $M$ 
has constant $(r,j)$-radical rank (respectively, $(r,j)$-socle rank) for all $j$.
\end{defn}

\begin{remark} For $r>1$, 
the condition that the $r$-radical type of $M$ is constant does not imply that the isomorphism
type of $\epsilon^*M$ is independent of $\epsilon \in \bE(r,\fg)$.    The condition that
$\dim \Rad^j(\epsilon^*(M)) = \dim \Rad^j(\epsilon^{\prime*}M)$ for all $j$ is much weaker
than the condition that $\epsilon^*M \simeq \epsilon^{\prime*}M$.  Indeed, examples are 
given in \cite{CFP2} (for $\fg = \fg_a^{\oplus n}$) of modules $M$ whose $r$-radical type
is constant but whose  $r$-socle type is not constant. In particular,  the isomorphism
type of $\epsilon^*M$ for such $M$ varies with $\epsilon \in \bE(r,\fg)$.
\end{remark}

\begin{prop}
\label{prop:el}
A $\fu(\fg)$-module 
$M$ has constant $(r,j)$-radical rank (respectively, $(r,j)$-socle rank) if and only if 
$\bR ad^j(r,\fg)_M = \emptyset$ (resp., $\bS oc^j(r,\fg)_M = \emptyset$.)
\end{prop}

\begin{proof}
This follows from the fact that there is a non-maximal radical rank if and only if
the radical rank is not constant, a non-minimal socle rank if and only if the socle rank is not
constant.
\end{proof}

\begin{prop}
\label{prop:single} 
 Let $G$ be an affine algebraic group, and let $\fg = \Lie(G)$.  
If $\bE(r,\fg)$ consists of a single $G$-orbit, then any finite dimensional
rational $G$-module has  constant $r$-radical type and constant $r$-socle type.
\end{prop}
\begin{proof} 
This follows immediately from Proposition \ref{prop:orbit}.
\end{proof}

\begin{remark}
We point out that examples arising from Proposition~\ref{prop:single} 
have much stronger properties than constant radical or socle rank:  they have 
the same isomorphism type restricted to any elementary subalgebra of dimension $r$. 
On the other hand, using $L_\zeta$-modules, we give examples in 
Propositions~\ref{prop:zeta},~\ref{prop:existzeta} of modules which 
have constant radical types but do not arise from a single $G$-orbit  
and don't even have $G$-structure.

\end{remark}

\begin{ex}
\label{ex:proj}
If $P$ is a finite dimensional projective $\fu(\fg)$-module, then $\epsilon^*P$ is a projective
(and thus free) $\fu(\epsilon)$-module for any elementary subalgebra $\epsilon \subset \fg$.
Thus, the $r$-radical type and $r$-socle type of $P$ are constant.
\end{ex}

\begin{ex}
\label{ex:heller}
Let $\fg$ be a $p$-restricted Lie algebra.  Recall that $\Omega^s(k)$
for $s > 0$ is  the kernel of $\xymatrix@=12pt{P_{s-1}\ar[r]^-d& P_{s-2}}$, where $d$ is the differential in the 
minimal projective resolution $\xymatrix@=12pt{P_* \ar[r]& k}$ of $k$ as a $\fu(\fg)$-module; if $s < 0$, 
then $\Omega^s(k)$ is the cokernel of $\xymatrix@=12pt{I^{-s-2} \ar[r]^-d& I^{-s-1}}$, 
where $d$ is the differential in the minimal injective resolution $\xymatrix@=12pt{k = I^{-1} \ar[r]& I^*}$ of $k$ as a $\fu(\fg)$-module.
Then for any $s \in \Z$, the $s$-th Heller shift $\Omega^s(k)$  has
constant $r$-radical type and constant $r$-socle type for each $r > 0$.  

Namely, for any $\epsilon \in \bE(r,\fg)$, $\epsilon^*(\Omega^s(k))$  is the 
direct sum of the $s$-th Heller shift of the trivial module $k$ 
and a free $\fu(\epsilon)$-module (whose rank is independent of the choice of 
$\epsilon \in \bE(r,\fg)$).  
\end{ex}

The following example is one of many we can realize using Proposition \ref{prop:single}.

\begin{ex}
\label{gl2n}
Let $\fg = \gl_{2n}$ and $r = n^2$. If $M$ is any finite dimensional
rational $\GL_{2n}$-module, then it has constant $r$-radical type and  constant $r$-socle type by
Corollary~\ref{cor:gln}.
\end{ex}

In Example \ref{gl2n}, the dimension $r$ of elementary subalgebras $\epsilon \subset \fg$ is
maximal.  We next consider an example of non-maximal elementary subalgebras. 

\begin{ex}\label{ex:nullext}
Choose $r > 0$ such that
no elementary subalgebra of dimension $r$ in $\fg$ is maximal.  Let $\zeta \in
\widehat{\HHH}^n(\fu(\fg),k)$ for $n < 0$ be an element in negative Tate cohomology.
Consider the associated short exact sequence 
\begin{equation}
\label{seqq}
\xymatrix{
 0 \ar[r] & k \ar[r] & E \ar[r] & \Omega^{n-1}(k) \ar[r] & 0.
}
\end{equation}
Then $E$ has constant $r$-radical rank and constant $r$-socle rank for
every $j, \ 1\leq j \leq (p-1)r$.

Namely, we observe that the restriction of the exact sequence \eqref{seqq} to $\epsilon$ splits  for every $\epsilon \in \bE(r,\fg)$.
  This splitting is a consequence of \cite[3.8]{CFP2}
(stated for an elementary abelian $p$-group and equally applicable to any elementary
subalgebra $\ff \subset \fg$ which strictly contains $\epsilon$).  The assertion is now proved
with an appeal to Example \ref{ex:heller}. 
\end{ex}

We next proceed to consider modules $L_\zeta$, adapting to the context of
$p$-restricted Lie algebras the results of \cite[\S 5]{CFP2}.

\begin{prop} (see \cite[5.5]{CFP2})
\label{prop:zeta}
Suppose that we have a non-zero cohomology 
class $\zeta \in \HHH^m(\fu(\fg),k)$ satisfying
the condition that 
$$Z(\zeta) \ \subset \ \cN_p(\fg) \ \subset \ \fg$$
does not contain a linear subspace of dimension $r$ for some $r \geq 1$.  
Then the $\fu(\fg)$-module $L_\zeta$ has constant $r$-radical type.
\end{prop}

\begin{proof}
Consider  $\epsilon \in \bE(r,\fg)$.
The hypothesis implies that $\epsilon$ is not contained in  $Z(\zeta)$.  Hence, $\zeta\downarrow_\epsilon
\in \HHH^m(\fu(\epsilon),k)$ is not nilpotent. Recall that $\HHH^*(\fu(\epsilon),k) \simeq \HHH^*(\bZ/p^{\times r}, k)$ is a tensor 
product of a symmetric and an exterior algebras on $r$ generators. Therefore, a non-nilpotent element is not a zero divisor. 
 Proposition 5.3 of \cite{CFP2} applied to $\epsilon$ implies that 
\begin{equation}
\label{eq:Lzeta}
\Rad(L_{\zeta\downarrow_\epsilon}) = \Rad(\Omega^n(\epsilon^*k)),
\end{equation}
where $\Omega^n(\epsilon^*k)$ is the $n$-th Heller shift of the trivial $\fu(\epsilon)$-module. 
We note that the statement and proof of \cite[Lemma 5.4]{CFP2} generalizes immediately to the map 
$\xymatrix@=12pt{\fu(\epsilon) \ar[r]& \fu(\fg)}$ yielding the statement that 
$\dim \Rad(\epsilon^*(L_\zeta)) - \dim \Rad(L_{\zeta\downarrow_\epsilon}) = 
\dim \Rad(\epsilon^*(\Omega^n(k)))  - \dim \Rad(\Omega^n(\epsilon^*k))$ 
is independent of $\epsilon$ whenever $\zeta\downarrow_\epsilon \not = 0$. 
Combined with \eqref{eq:Lzeta}, this allows us to conclude that 
\sloppy{

}
\[
\dim \Rad(\epsilon^*(L_\zeta)) = \dim \Rad(\epsilon^*(\Omega^n(k))).
\]
 
Since $\epsilon^*(L_\zeta)$ is a submodule of $\epsilon^*(\Omega^n(k))$ this further implies that equality of radicals 
\[
\Rad^j(\epsilon^*(L_\zeta)) = \Rad^j(\epsilon^*(\Omega^n(k)))
\] 
for all $j>0$. Since $\Omega^n(k)$ has constant $r$-radical type  by 
Example~\ref{ex:heller}, we conclude that the same holds for $L_\zeta$. 
\sloppy{

}
\end{proof}
Utilizing another result of \cite{CFP2}, we obtain a large class of $\fu(\fg)$-modules of 
constant radical type.
\begin{prop}\label{prop:existzeta}
Let $d$ be a positive integer, sufficiently large compared to $r$ and $\dim \fg$. There exists some $0 \not= \zeta \in \HHH^{2d}(\fu(\fg),k)$
such that $L_\zeta$ has constant $r$-radical type.
\end{prop}

\begin{proof}
The embedding $V(\ul \fg) \simeq \Spec \HHH^{\rm ev}(\fu(\fg),k) \ \hookrightarrow \ \fg$ (for $p>2$) is given by the natural 
map $\xymatrix@=12pt{S^*(\fg^\#[2]) \ar[r]& \HHH^*(\fu(\fg),k)}$ determined by the Hochschild construction 
$\xymatrix@=12pt{\fg^\# \ar[r]& \HHH^2(\fu(\fg),k)}$ (see, for example, \cite{FPar1}).  (Here, $\fg^\#[2]$ is the 
vector space dual to the underlying vector space of $\fg$, placed in cohomological 
degree 2.) As computed in \cite[5.7]{CFP2}, the set of all homogeneous polynomials 
$F$ of degree $d$ in $S^*(\fg^\#[2])$
 such that the zero locus $Z(F) \subset \Proj(\fg)$ does not contain a linear hyperplane 
isomorphic to $\bP^{r-1}$ is dense in the space of all polynomials of degree $d$ for $d$ sufficiently large.    
Let $\zeta$ be the restriction to  $\Proj k[V(\ul\fg)]$ of such an $F$ in $S^*(\fg^\#[2])$; 
since such an $F$ can be chosen from a dense subset  of
homogeneous polynomials of degree $d$, we may find such an $F$ whose associated
restriction $\zeta$ is non-zero.  Now, we may apply Proposition \ref{prop:zeta} to 
conclude that $L_\zeta$ has constant $r$-radical type.
\end{proof}

The following closure property for modules of constant radical and socle types is an 
extension of a similar property for modules of constant Jordan type.

\begin{prop} Suppose $\bE(r, \fg)$ is connected. 
Let  $M$ be a $\fu(\fg)$-module of constant $(r,j$)-radical rank (respectively, constant 
$(r,j$)-socle rank) for some $r, j$.  Then any $\fu(\fg)$-summand $M^\prime$ of $M$ also has 
constant $(r,j$)-radical rank  (resp., constant $(r,j$)-socle rank).
\end{prop}

\begin{proof}
Write  $M = M^\prime \oplus M^{\prime\prime}$, and set $m$ equal to the $(r,j)$-radical 
rank of $M$.  Since the local  $(r,j)$-radical types of 
$M^\prime, \  M^{\prime\prime}$ are both lower semicontinuous by Theorem \ref{upper-lower} and
since the sum of these local radical types is a constant function, we conclude that both  
$M^\prime, \  M^{\prime\prime}$ have constant $(r,j)$-radical rank.

The argument for $(r,j)$-socle rank is essentially the same.
\end{proof}

\end{document}